\documentclass[fleqn,10pt]{article}
\usepackage{amsmath}
\usepackage{amssymb}
\usepackage{amsthm}
\usepackage[authoryear]{natbib}
\usepackage{graphicx}
\usepackage[colorlinks=true, urlcolor=citecolor, linkcolor=citecolor, citecolor=citecolor]{hyperref}
\usepackage[a4paper]{geometry}
\usepackage[english]{babel}
\usepackage{placeins}

  \providecommand{\assumptionname}{Assumption A$\!\!\!$}
  \providecommand{\definitionname}{Definition}
  \providecommand{\lemmaname}{Lemma}
  \providecommand{\propositionname}{Proposition}
  \providecommand{\remarkname}{Remark}
\providecommand{\theoremname}{Theorem}
\providecommand{\corollaryname}{Corollary}

  \theoremstyle{plain}
  \newtheorem{thm}{\protect\theoremname}[section]
  \newtheorem{assumption}{\protect\assumptionname}
  \newtheorem{defin}{\protect\definitionname}[section]
  \newtheorem{prop}{\protect\propositionname}[section]
  
  \newtheorem{cor}{\protect\corollaryname}[section]
  \theoremstyle{remark}
  \newtheorem{rem}{\protect\remarkname}[section]

\newcommand{\Var}{\mathbf{V}{\rm ar}}

\begin{document}

\title{\textsc{Detection of dependence patterns with delay}}

\author{Julien Chevallier\footnote{Corresponding author: {\sf{e-mail: julien.chevallier@unice.fr}}} , ~Thomas Lalo\"e\\
Univ. Nice Sophia Antipolis, CNRS, LJAD, UMR 7351, 06100 Nice, France.
}
\date{} 

\maketitle              

\begin{abstract}
{The Unitary Events (UE) method is a popular and efficient method used this last decade to detect dependence patterns of  joint spike activity among simultaneously recorded neurons. The first introduced method is based on binned coincidence count \citep{Grun1996} and can be applied on two or more simultaneously recorded neurons. Among the improvements of the methods, a transposition to the continuous framework has recently been proposed in \citep{muino2014frequent} and fully investigated in \citep{MTGAUE} for two neurons. The goal of the present paper is to extend this study to more than two neurons. The main result is the determination of the limit distribution of the coincidence count. This leads to the construction of an independence test between $L\geq 2$ neurons. Finally we propose a multiple test procedure via a Benjamini and Hochberg approach \citep{Benjamini1995}. All the theoretical results are illustrated by a simulation study, and compared to the UE method proposed in \citep{Grun2002}. Furthermore our method is applied on real data.}
\end{abstract}



\setcounter{tocdepth}{1}
\tableofcontents{}

\clearpage

\section{Introduction}
\label{sec1}

The communication between neurons relies on their capacity to generate characteristic electric pulses called action potentials. These action potentials are usually assumed to be identical stereotyped events. Their time of occurrence (called spike) is considered as the relevant information. That is why the study of spike frequencies (firing rates) of neurons plays a key role in the comprehension of the information transmission in the brain \citep{Abeles1982,Gerstein1969,Shinomoto2010}. Such neuronal signals are recorded from awake behaving animals by insertion of electrodes into the cortex to record the extracellular signals. Potential spike events are extracted from these signals by threshold detection and, by spike sorting algorithms, sorted into the spike signals of the individual single neurons. After this preprocessing, we dispose of sequences of spikes (called  spike trains).\\

\begin{sloppypar}The analysis of spike trains has been an area of very active research for many years \citep{Brown}. Although the rules underlying the information processing in the brain are still under burning debate, the detection of correlated firing between neurons is the objective of many studies in the recent years \citep{Roy2007,Dong2008,Pillow_nature}. This synchronization phenomenon may take an important role in the recognition of sensory stimulus.  In this article, the issue of detecting dependence patterns between simultaneously recorded spike trains is addressed. Despite the fact that some studies used to consider neurons as independent entities \citep{Barlow1972}, many theoretical works consider the possibility that neurons can coordinate their activities \citep{Hebb1949,Palm1990,Sakurai1999,VonDerMalsburg1981}. The understanding of this synchronization phenomenon \citep{Singer1993} required the development of specific descriptive analysis methods of spike-timing over the last decades: cross-correlogram \citep{Perkel1967}, gravitational clustering \citep{Gerstein1985} or joint peristimulus time histogram (JPSTH, \citep{Aertsen89}).  Following the idea that the influence of a neuron over others (whether exciting or inhibiting) results in the presence (or absence) of coincidence patterns, Gr\"un and collaborators developed one of the most popular and efficient method used this last decade: the Unitary Events (UE) analysis method \citep{Grun1996} and the corresponding independence test, which detects where dependence lies by assessing p-values (A Unitary Event is a spike synchrony that recurs more often than expected by chance). This method is based on a binned coincidence count that is unfortunately known to suffer a loss in synchrony detection, but this flaw has been corrected by the multiple shift coincidence count \citep{Grun1999}. 

In order to deal with continuous time processes, a new method ( Multiple Tests based on a Gaussian Approximation of the Unitary Events method), based on a generalization of this count, the delayed coincidence count, has recently been proposed for two parallel neurons (Section 3.1 of \citep{MTGAUE}). The results presented in this article are in the lineage of this newest method and are applied on continuous point processes (random set of points which are modelling spike trains). Testing independence between real valued random variables is a well known problem, and various techniques have been developed, from the classical chi-square test to re-sampling methods for example. The interested reader may look at \citep{Lehm_2005_book}. Some of these methods and more general surrogate data methods have been applied on binned coincidence count, since the binned process transforms the spike train in vectors of finite dimension. However, the case of point processes that are not preprocessed needs other tools and remains to study. Although the binned method can deal with several neurons (six simultaneously recorded neurons are analysed in \citep{Grun2002}, both of the improvements (Multiple Shift and MTGAUE) can only consider pairs of neurons. Thus, our goal is to generalize the method introduced in \citep{MTGAUE} for more than two neurons. Unlike MTGAUE, our test is not designed to be performed on multiple time windows. However it can be multiple with respect to the different possible patterns composed from $n\geq 2$ neurons (see Section \ref{Mpt}). \\\end{sloppypar}

In Section \ref{sec:Coincidence}, we introduce the different notions of coincidence used through this article. In Section \ref{sec:Test}, a test is established and the asymptotic control of its false positive rate is proven. In Section \ref{sec:Simulations} our test is confronted to the original UE method on simulated data and the accuracy of the Gaussian approximation is verified. In Section \ref{sec:Hawkes} the relevance of our method when our main theoretical assumptions are weakened is also empirically put on test. Section \ref{sec:real} presents an illustration on real data. All the technical proofs are given in the Appendix.

\section{Notions of coincidence and the classical UE methods}
\label{sec:Coincidence}

In order to detect synchronizations between the involved neurons, different notions of coincidence can be considered. Informally, there is a coincidence between neurons when they each emit a spike more or less simultaneously. This notion has already been used in UE methods \citep{Grun2002} and is based on the following idea: a real dependency between $n\geq 2$ neurons should be characterized by an unusually large (or low) number of coincidence \citep{Grammont2003,Grun1996,MTGAUE}.\\

\subsection{Two notions of coincidence}

The UE method (see \citep{Grun1996}) considers discretized spike trains at a resolution $\ell$ of typically 1 or 0.1 millisecond. Therefore, in the discrete-time framework, each trial consists of a set of $n$ spike trains (one for each recorded neuron), each spike train being represented by a sequence of $0$ and $1$ of length $S$. Since it is quite unlikely that two spikes occur at exactly the same time at this resolution $\ell$, spike trains are binned and clipped at a coarser level. More precisely for a fixed bin size $\Delta=d\ell$ ($d$ being an integer), a new sequence of length $S/d$ of $0$ and $1$ is associated to each spike train ($1$ if at least one spike occurs in the corresponding bin, $0$ otherwise). For more precise informations on the binning procedure and the link with point processes we refer the interested reader to \citep{MTGAUE}. \\

\begin{figure}[!h]
\begin{centering}
\includegraphics[width=0.8\textwidth]{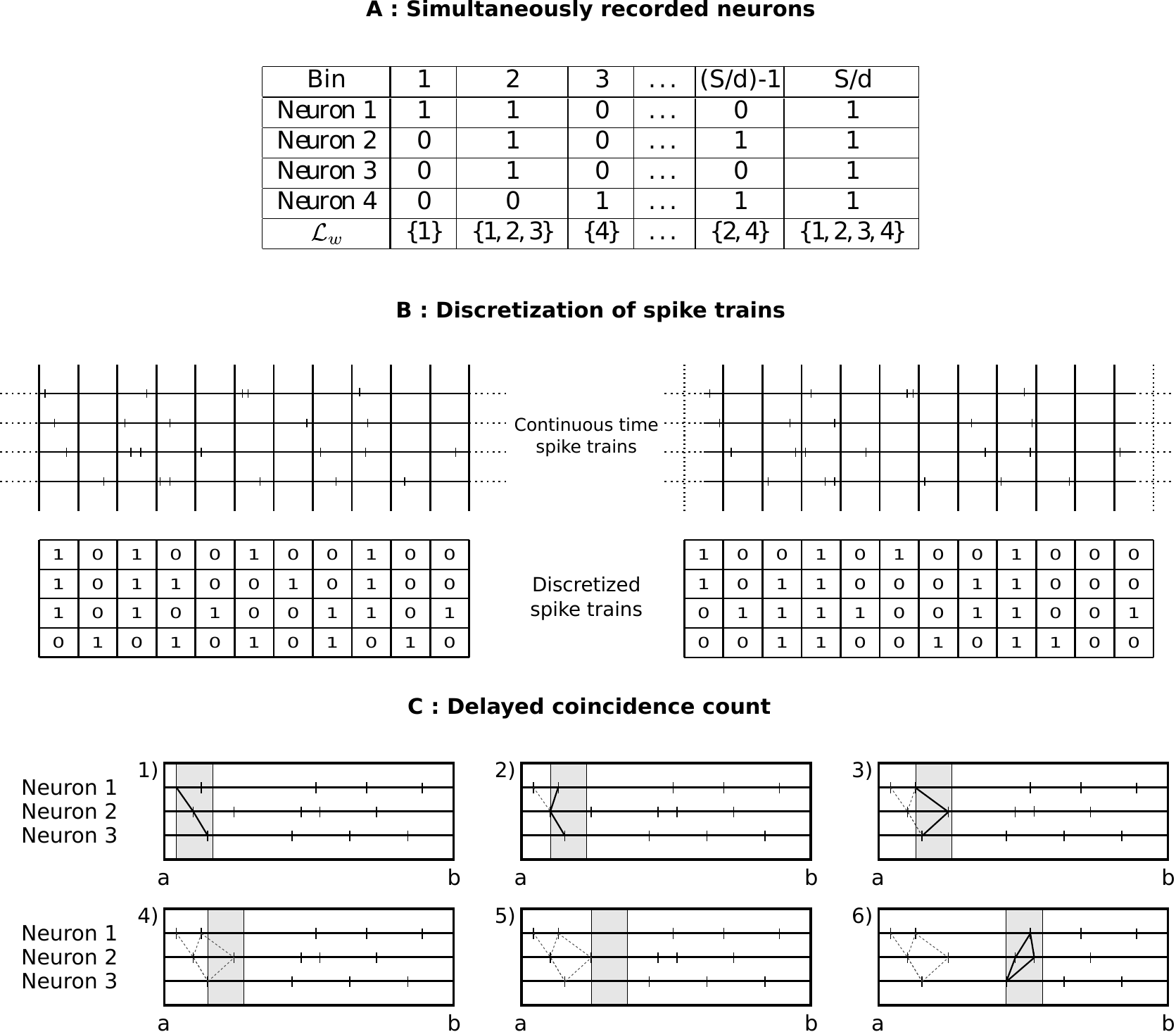}
\par
\end{centering}

\caption{\label{fig: Coincidence}
\small
In \textbf{A}, $4$ parallel binary processes of length $S$ are displayed. At each time step, the constellation and its corresponding subset of $\{1,2,3,4 \}$ are given. For instance, the constellation associated to the first bins is the vector $(1,0,0,0)$ and the corresponding subset is $\{1\}$. In \textbf{B}, illustration of the UE method with two different choices of bins of the same size (the results are different, for example the constellation full of $1$s is present in the second case and not in the first one). In \textbf{C}, an illustration of the six first steps in the dynamical computation of the delayed coincidence count. Here, there are $3$ parallel  time point processes. We consider the full pattern, i.e. $\mathcal{L}=\{1,2,3\}$. The grey rectangle represents the sliding time window of length $\delta$. The bold lines denote the coincidence patterns counted at each step and the grey dashed ones denote the coincidence pattern which have been counted in one the previous steps. At each of steps $1$,$2$ and $3$, exactly one coincidence is counted. At steps $4$ and $5$, no coincidence is detected. And, at step $6$, two coincidences are counted.
}
\end{figure}

A constellation or pattern is a vector of size $n$ of $0$ and $1$ (see Figure \ref{fig: Coincidence} or \citep{Grun2002}). Of course, there are $2^n$ different constellations. The UE statistic associated to some constellation $w$ consists in counting the number of occurrences of such $w$ in the set of $S/d$ vectors of size $n$.\\

However, as shown in Figure \ref{fig: Coincidence}, this method largely depends on the bin choice and it has been  proven in \citep{Grun1999} that this can lead in the case $n=2$ to up to $60\%$ of loss in detection when $\Delta$ is of the order of the range of interaction.\\

Then, we focus on another coincidence count that deals with continuous data. This notion of delayed coincidence count is pretty natural and was used in \citep{muino2014frequent} or \citep{borgelt2013finding} in a simplified formalism. For sake of simplicity, we use the same formalism of point processes as in \citep{MTGAUE}. Nevertheless, we give the correspondences, whenever it is possible, between their formalism and ours (see Table \ref{table:Correspondences}).\\

Considering $N_{1},\dots,N_{n}$, some point processes on $\left[a,b\right]$, and $\mathcal{L}\subset \{1,\dots,n\}$ a set of indices of cardinal $L\geq 2$,  the \emph{delayed coincidence count} $X_{\mathcal{L}}$ (of delay $\delta<(b-a)/2$) over the neurons of subset $\mathcal{L}$ in the time window $[a,b]$ is given by

\begin{equation}
\label{eq:count}
X_{\mathcal{L}}=X_{\mathcal{L}}(\delta)=\sum_{(x_1,\dots,x_L)\in \prod_{l\in \mathcal{L}} N_{l}}
\mathbf{1}_{\left| \max\limits_{i\in\left\{1,\dots ,L\right\}}x_{i} -\min\limits_{i\in\left\{1,\dots ,L\right\}}x_{i} \right|\leq\delta}.
\end{equation}

The  delayed coincidence count can be explained in the following way (see Figure \ref{fig: Coincidence}):
\begin{itemize}
\item
Fix some duration parameter $\delta$ which is the equivalent of the bin size $\Delta$,
\item Count how many times each neuron in $\mathcal{L}$ spikes almost at the same time, modulo the delay $\delta$.
\end{itemize}

\begin{table}[h]

\begin{centering}
\begin{small}
\begin{tabular}{|c|l|l|l|}

   \hline
    & \citep{muino2014frequent} & \citep{borgelt2013finding} & This article \\
   \hline 
   Time window & $[0,T]$ & Not relevant & $[a,b]$ \\
   Subset & $A$ & $A$ & $\mathcal{L}$\\
   Delay & $r$ & $w$ & $\delta$\\
   Number of coincidences & Not relevant & $|\mathcal{E}(A,w)|$ (page 3) & $X_{\mathcal{L}}$\\
   Spike-train synchrony & $Supp_{|A|}(E^{A})$ & $|\mathcal{E}(A,w)|$ (page 4) & Not relevant\\
   \hline
\end{tabular}
\par
\end{small}
\end{centering}
\caption{\label{table:Correspondences} List of the correspondences in the notations between three papers: \citep{muino2014frequent}, \citep{borgelt2013finding} and this article.}
\end{table}

\begin{rem}
For computational reasons, Mui\~no and Borgelt compute some non-overlapping coincidence count, which corresponds to the last line of the Table \ref{table:Correspondences}. They impose the condition that at most one coincidence per spike is counted. The statistical study of this non-overlapping coincidence count is not the scope of this article. Although interesting, this is a much more challenging question.
\end{rem}

\subsection{Original UE method}

The notion of constellation is closely linked to the binning procedure and is not relevant in the continuous time framework. In this work, we fix some subset of neurons, denoted $\mathcal{L}$, and count how many times the neurons of $\mathcal{L}$ admit nearly simultaneous activity. However, there is a canonical correspondence between constellations and set of indices (see Figure \ref{fig: Coincidence}). Then, in order to harmonize the notations between both methods, let us denote $\mathcal{L}(w)$ the set of indices corresponding to the constellation $w$.\\

To detect dependency between neurons, two estimators of the expected coincidence count are compared. The first one is the empirical mean $\bar{m}_w$ of the number of occurrences of a given constellation $w$ through $M$ trials,
\[\bar{m}_w=\frac{1}{M} \sum_{k=1}^{M} m_{w}^{(k)}, \]
where $m_{w}^{(k)}$ is the number of occurrences of $w$ during the $k^{th}$ trial. This estimator is consistent (that is, converges towards the expected value of the number of occurrences) even with dependency between the spike trains. The second one is consistent only under the independence hypothesis, and is given by

\begin{equation}
\hat{m}_{g,w}=\frac{S}{d} \prod_{l\in\mathcal{L}(w)} \hat{p}_l \prod_{k\notin\mathcal{L}(w)} (1-\hat{p}_k),
\end{equation}

\noindent where $\hat{p}_i$ is the empirical probability of finding a spike in a bin of neuron $i$. \\

This enables the construction of the test described in \citep{Grun2002} and based on the comparison between the statistic $M\bar{m}_w$ and a quantile of the Poisson distribution $\mathcal{P}(M\hat{m}_{g,w})$ where $M$ is the number of trials. Most of the time only tests by upper values are computed \citep{Grun1996,Grun2002}.  However, following the study of \citep{MTGAUE}, we have decided to focus on symmetric tests. Hence, the symmetric test based on the UE method rejects the independence hypothesis when $\bar{m}_w$ is too different from $\hat{m}_{g,w}$. However, such a test necessarily makes mistakes. For example, a \emph{false positive} corresponds to an incorrect rejection of the null hypothesis. Hence, an a priori upper bound on the false positive rate, that is the \emph{significance level} (or just \emph{level}), must be given in order to construct a decision rule. The symmetric independence test with level $\alpha$ based on the UE method is governed by the following rule:   if 
\[M\bar{m}_{w} \geq  q_{1-\alpha/2} \quad \mbox{or} \quad M\bar{m}_{w} \leq  q_{\alpha/2},\]
where $q_x$ is the $x$-quantile of the Poisson distribution $\mathcal{P}(M\hat{m}_{g,w})$, then the independence hypothesis is rejected.\\

The UE method is applied under the hypothesis that the discrete processes modelling the spike trains of neurons are  Bernoulli processes. The equivalent in the "continuous" framework is the Poisson process (as it can be seen in \citep{MTGAUE}). This leads to a different estimator of the expected coincidence count and a different test which are defined properly in the next section.

\section{Study of the delayed coincidence count}
\label{sec:Test}

Once the notion of coincidence is defined with respect to continuous data (Equation \eqref{eq:count}) , mathematical tools can be used to construct the desired independence test. The procedure is to provide the expected value and variance of the variable $X_{\mathcal{L}}$ in function of the firing rates. These computations classically imply a Gaussian approximation with respect to i.i.d trials. Unfortunately the firing rates are usually unknown. Thus the final step is to replace the firing rates by their estimator to compute the estimated expected value and variance. This plug-in procedure is known to change the underlying distribution. As in \citep{MTGAUE}, the delta method provides the exact nature of this change.\\

In the continuous framework, a sample is composed of $M$ observations of $N_{1},\cdots,N_{n}$  which are the point processes associated to the spike trains of $n$ neurons on a window $\left[a,b\right]$. The goal is to answer the following question: 
\begin{center}
{\it Given $\mathcal{L}$ a subset of $\left\{1,\dots ,n\right\}$, are the processes $N_{l},\ l\in\mathcal{L}$ independent?}
\end{center}

To do this, a statistical test comparing the two hypotheses
\[
\begin{cases}
\left(\mathcal{H}_{0}\right) & \mbox{The processes \ensuremath{N_{l},\ l\in\mathcal{L}} are independent;}\\
\left(\mathcal{H}_{1}\right) & \mbox{The processes \ensuremath{N_{l},\ l\in\mathcal{L}} are not independent;}
\end{cases}
\]
is proposed.\\

In this section our test and its asymptotic relevance are introduced. First, let us present and discuss our main assumptions which are the same as in \citep{MTGAUE}.

\begin{assumption}
\label{ass: Poisson}$N_{1},\dots,N_{n}$ are Poisson processes.
\end{assumption}

This assumption can be resumed to an assumption of independence of a point process with respect to itself over the time, as Bernoulli processes in discrete settings. 

\begin{assumption}
\label{ass: Partition homogene} The Poisson processes $N_{1},\dots,N_{n}$ are homogeneous on $[a,b]$. 
\end{assumption}
Assumption A\ref{ass: Partition homogene} may also appear very restrictive. But once again Bernoulli processes considered in \citep{Grun1999,Grun2002} have the same drawback. Moreover, if necessary, one can partition $[a,b]$ in smaller intervals on which A\ref{ass: Partition homogene} is satisfied. For more precise informations on Poisson processes we refer the interested reader to \citep{Poisson_Process}.\\

These assumptions are necessary in this work in order to obtain an explicit form for the expected number of coincidences (and its variance). Note that there exist some  surrogate methods in the literature for which there is no need of a model on the data (see \citep{grun2009data,louis2010generation} for a review). In particular  two kind of methods are commonly used: dithering methods (involving random shifts of individual spikes \citep{stark2009unbiased,louis2010surrogate}, or random shifts of patterns of spikes \citep{harrison2009rate}), and trial-shuffling methods \citep{Pipa2003,PipaGrun}. However, they are based on binned coincidence count, and there is no equivalent, up to our knowledge, with a delayed coincidence count, due to serious computational issues. Alternative works have also been done in the Bayesian paradigm \citep{Archer2013}. However, as announced in the introduction, we empirically show in Section \ref{sec:Hawkes} that the assumptions can be weakened. In particular, point processes admitting refractory periods can be taken into account. Thus, a nice perspective of this work could be to derive theoretical results with these weakened assumptions.\\

\subsection{Asymptotic properties}
\label{sub:Asymptotic:properties}

In order to build our independence test, one needs to understand the behaviour of the number of coincidence $X_{\mathcal{L}}$ under the independence hypothesis $\mathcal{H}_0$. In particular, the expected value and the variance of $X_{\mathcal{L}}$ are computed here. In a general point processes framework, these computations are impossible. This is why some restrictive assumptions are needed, such as A\ref{ass: Poisson}, A\ref{ass: Partition homogene}, or the independence of the processes, as done in the original UE method where independent Bernoulli processes have been considered.

\begin{thm}
\label{thm:Moyenne et Esperance} Let $\mathcal{L}$ and $X_{\mathcal{L}}$ be defined as previously. Assume assumptions A\ref{ass: Poisson} and A\ref{ass: Partition homogene} and denote by $\lambda_{1},\dots,\lambda_{n}$ the respective intensities of $N_{1},\dots,N_{n}$. Under hypothesis $\mathcal{H}_{0}$, the expected value and the variance of the number of coincidences $X_{\mathcal{L}}$ are given by:

\[
m_{0, \mathcal{L}}:=\mathbb{E}\left[X_{\mathcal{L}}\right]=\left(\prod_{ l\in\mathcal{L}}\lambda_{ l}\right)I(L,0)
\]
and 
\[
\Var(X_{\mathcal{L}})=m_{0, \mathcal{L}}+\sum_{k=1}^{L-1}\left(\sum_{\begin{subarray}{c}
\mathcal{J}\subset\mathcal{L}\\
\#\mathcal{J}=k
\end{subarray}}\prod_{j\in\mathcal{J}}\lambda_{j}^{2}\prod_{l\notin\mathcal{J}}\lambda_{l}\right)I(L,k)
,\]
where the $I(L,k)$ are given by Proposition \ref{prop:Formule des I(k)} below.
\end{thm}

The proof relies on the calculus of the moments of a sum over a Poisson Process and is given in Appendix \ref{app: expectation and variance}. The integral $I(L,k)$ can be seen as the contribution of a subset of $k$ neurons to the number of coincidences between the $L$ neurons.

\begin{prop}
\label{prop:Formule des I(k)}
For $b>a\geq0$ and $0<\delta<b-a$, define for every $k$ in $\left\{0,\dots ,L\right\} $
\[
I(L,k)=\intop_{[a,b]^{L-k}}\left(\intop_{[a,b]^{k}}\mathbf{1}_{\left|\max\limits_{i\in\left\{1,\dots ,L\right\}}x_{i} -\min\limits_{i\in\left\{1,\dots ,L\right\}}x_{i} \right|\leq\delta}\, dx_{1}\ldots dx_{k}\right)^{2}dx_{k+1}\ldots dx_{L},
\]
where the convention $\intop\limits _{[a,b]^{0}}f\left(x\right) dx=f\left(x\right)$ is set.
Then, for $L\geq2$, and $k$ in $\left\{0,\dots ,L-1\right\} $,
\begin{itemize}
\item $I\left(L,L\right)=L^{2}\left(b-a\right)^{2}\delta^{2L-2}-2L\left(L-1\right)\left(b-a\right)\delta^{2L-1}+\left(L-1\right)^{2}\delta^{2L}$,

\item $I\left(L,k\right)=f\left(L,k\right)\left(b-a\right)\delta^{L+k-1}-h\left(L,k\right)\delta^{L+k}$,

where $\displaystyle f\left(L,k\right)=\frac{k\left(k+1\right)+L\left(L+1\right)}{L-k+1},$

and $\displaystyle h\left(L,k\right)=\frac{-k^{3}+k^{2}(2+L)+k(5+2L-L^{2})+L^{3}+2L^{2}-L-2}{(L-k+2)(L-k+1)}$.

\end{itemize}
\end{prop}

Once the behaviour of $X_{\mathcal{L}}$ under $\mathcal{H}_0$ is known, the method to construct an independence test is straight-forward. Suppose that $M$ independent and identically distributed (i.i.d.) trials are given. Denote $N_{i}^{\left(k\right)}$ the spike train corresponding to the neuron $i$ during the $k^{th}$ trial. As for the UE method, the idea is to compare two estimates of the expectation of $X_{\mathcal{L}}$. The first one is the empirical mean of $X_{\mathcal{L}}$:

\begin{equation}
\bar{m}_{\mathcal{L}}=\frac{1}{M}\sum_{k=1}^{M}X_{\mathcal{L}}^{\left(k\right)},
\end{equation}
where $X_{\mathcal{L}}^{\left(k\right)}$ is the delayed coincidence count during the $k^{th}$ trial. This estimate converges even if the processes are not independent. More precisely the following asymptotic result is given by the Central Limit Theorem
\[ \sqrt{M}\left(\bar{m}_{\mathcal{L}}- \mathbb{E}\left[X_{\mathcal{L}}\right] \right) \underset{M\rightarrow\infty}{\overset{\mathcal{D}}{\longrightarrow}}\mathcal{N}\left(0,\Var(X_{\mathcal{L}})\right) ,\]
where $\overset{\mathcal{D}}{\longrightarrow}$ denotes the convergence of distribution  and $\mathcal{N}(\mu,\sigma^2)$ denotes the Gaussian distribution with mean $\mu$ and variance $\sigma^2$.

The second estimate is given by Theorem \ref{thm:Moyenne et Esperance}. Indeed, under $\mathcal{H}_0$ the following equality holds
\[ \mathbb{E}\left[X_{\mathcal{L}}\right] =m_{0,\mathcal{L}}=\left(\prod_{ l\in\mathcal{L}}\lambda_{ l}\right)I(L,0).\] 
 Replacing each spiking intensity $\lambda_{l}$ by 
\[\hat{\lambda}_{l}:=\frac{1}{M(b-a)}\sum_{k=1}^{M} N_{l}^{(k)}\left([a,b]\right),\]
where $N_{l}^{(k)}\left([a,b]\right)$ denotes the number of spikes in $[a,b]$ for neuron $l$ during the $k^{th}$ trial, gives the following estimator,

\begin{equation}
\hat{m}_{0,\mathcal{L}}=\left(\prod_{ l\in\mathcal{L}}\hat{\lambda}_{ l}\right)I\left(L,0\right).
\end{equation}

Note that $\bar{m}_{\mathcal{L}}$ is always consistent (that is, converges towards the true parameter) whereas $\hat{m}_{0,\mathcal{L}}$ is consistent only under $\mathcal{H}_{0}$. This leads to the following independence test: the independence assumption is rejected when the difference between $\bar{m}_{\mathcal{L}}$ and $\hat{m}_{0,\mathcal{L}}$ is too large. More precisely, Theorem \ref{thm:Delta Methode} gives the asymptotic behaviour of $\sqrt{M}\left(\bar{m}_{\mathcal{L}}-\hat{m}_{0,\mathcal{L}}\right)$ under $\mathcal{H}_{0}$.

\begin{thm}
\label{thm:Delta Methode}Under the notations and assumptions of Theorem \ref{thm:Moyenne et Esperance}, and  under $\mathcal{H}_{0}$, the following affirmations are true 
\begin{itemize}
\item The following convergence of distribution holds:
\[
\sqrt{M}\left(\bar{m}_{\mathcal{L}}-\hat{m}_{0,\mathcal{L}}\right)\underset{M\rightarrow\infty}{\overset{\mathcal{D}}{\longrightarrow}}\mathcal{N}\left(0,\sigma^{2}\right),
\]
with
\[
\sigma^{2}=\Var(X_{\mathcal{L}})-(b-a)^{-1}\mathbb{E}\left[X_{\mathcal{L}}\right]^{2}\left(\sum_{l\in\mathcal{L}}\lambda_{ l}^{-1}\right).
\]

\item Moreover, $\sigma^{2}$ can be estimated by 
\[
\hat{\sigma}^{2}=\hat{v}\left(X_{\mathcal{L}}\right)-(b-a)^{-1}I(L,L)\prod_{ l\in\mathcal{L}}\hat{\lambda}_{ l}^{2}\left(\sum_{ l\in\mathcal{L}}\hat{\lambda}_{ l}^{-1}\right),
\]
where
\[
\hat{v}(X_{\mathcal{L}})=\hat{m}_{0,\mathcal{L}} +\sum_{k=1}^{L-1}
\left(
\sum_{\begin{subarray}{c}
\mathcal{J}\subset\mathcal{L}\\
\#\mathcal{J}=k
\end{subarray}}
\prod_{j\in\mathcal{J}}\hat{\lambda}_{j}^{2} \prod_{l\notin\mathcal{J}}\hat{\lambda}_{l}
\right)
I(L,k),
\]
and the following convergence of distribution holds:
\[
\sqrt{M}\frac{\left(\bar{m}_{\mathcal{L}}-\hat{m}_{0,\mathcal{L}}\right)}{\sqrt{\hat{\sigma}^{2}}}\overset{\mathcal{D}}{\rightarrow}\mathcal{N}\left(0,1\right).
\]

\end{itemize}
\end{thm}

The proof  of this theorem relies on a standard application of the delta method \citep{Casella2002} and is given in Appendix \ref{app: delta method}.  The delta method is useful in order to deal with the plug-in step, i.e. the substitution of the real parameters by the estimated ones.

Note that the results obtained in Theorems \ref{thm:Moyenne et Esperance}
and \ref{thm:Delta Methode} are true for more general delayed coincidence counts. A more general result and its proof are given in Appendix. However when one considers more general ways to count coincidences the integrals $I(L,k)$ are harder to compute.\\

\subsection{ Independence test}

The results obtained in Theorem \ref{thm:Delta Methode} allow us to straightforwardly build a test for detecting a dependency between neurons:

\begin{defin}[The GAUE test]
\label{test}
For $\alpha$ in $\left]0,1\right[$, denote $z_{\alpha}$ the $\alpha$-quantile of the standard Gaussian distribution $\mathcal{N}(0,1)$. Then the symmetric test of level $\alpha$ rejects $\mathcal{H}_0$ when $\bar{m}$ and $\hat{m}_{0,\mathcal{L}}$ are too different, that is when \[\left| \sqrt{M}\frac{\left(\bar{m}_{\mathcal{L}}-\hat{m}_{0,\mathcal{L}}\right)}{\sqrt{\hat{\sigma}^{2}}}\right| >  z_{1-\alpha/2}.\]

\end{defin}
Note that once a subset is rejected by our test, one can determine if the dependency is rather excitatory or inhibitory according to the sign of $\bar{m}_{\mathcal{L}}-\hat{m}_{0,\mathcal{L}}$. If $\bar{m}_{\mathcal{L}}-\hat{m}_{0,\mathcal{L}}> 0$ (\textit{respectively $< 0$}) then the dependency is rather excitatory (\textit{respectively inhibitory}).\\

The result of a test may be wrong in two distinct manners. On the one hand, a false positive is an error in which the test is incorrectly rejecting the null hypothesis. On the other hand, a \emph{false negative} is an error in which the test is incorrectly accepting the null hypothesis. The false positive (respectively negative) rate is the test's probability that a false positive (resp. negative) occurs. Usually, a theoretical control is given only for the false positive rate which is considered as the worst error. The following corollary is an immediate consequence of Theorem \ref{thm:Delta Methode} and states the appropriateness of the GAUE test.

\begin{cor}\label{cor:good:level}
Under assumptions of Theorem \ref{thm:Delta Methode}, the test of level $\alpha$ presented in Definition \ref{test} is asymptotically of false positive rate $\alpha$. 
That is, the false positive rate of the test tends to $\alpha$ when the sample size $M$ tends to infinity.
\end{cor}

\newpage

\section{Illustration Study: Poissonian Framework}
\label{sec:Simulations}

In this section, an illustration of the previous theoretical results is given. To obtain a global evaluation of the performance of the different methods, some parameters can randomly fluctuate. More precisely, the following procedure is applied,

$\left. \text{\parbox{0.9\linewidth}{
\begin{itemize}
\item[1.] Generate a set of random parameters according to the appropriate Framework;
\item[2.] Use this set to generate $M$ trials;
\item[3.] Compute the different statistics;
\item[4.] Repeat steps 1 to 3 a thousand times.
\end{itemize}
}} \right \}$ {\bf $\textrm{P}$}

We begin by an illustration of the results of Theorem \ref{thm:Delta Methode} and Corollary \ref{cor:good:level}, and a comparison with the original UE method.

\subsection{Illustration of the asymptotic properties}
\label{sec:simu}

The control on the false positive rate of our test being only asymptotic, it is evaluated on simulations in this Section. Moreover, it is shown that our test is empirically \emph{conservative}, that is, when constructed for a prescribed level, say $\alpha$, the empirical false positive rate is less than $\alpha$. We simulate independent Poisson processes under the following Framework ({\bf $\textrm{F}_1$}) :

$\left. \text{\parbox{0.9\linewidth}{
\begin{itemize}
\item the trial duration ($b-a$) is randomly selected (uniform distribution) between $0.2$s and $0.4$s;
\item the $n=4$ neurons are simulated with different intensities. Each one is randomly selected (uniform distribution) between 8 and 20Hz;
\item the set of tested neurons is given by ${\cal L}=\left\{ 1,2,3,4\right\}$;
\end{itemize}
}} \right \}$ {\bf $\textrm{F}_1$}

Moreover, we set once and for all $\delta=0.01s$. Note that the dependence with respect to the parameter $\delta$ has been fully discussed in \citep{Albert2014}.\\

\begin{figure}[h]
\centering{}\includegraphics[scale=0.5]{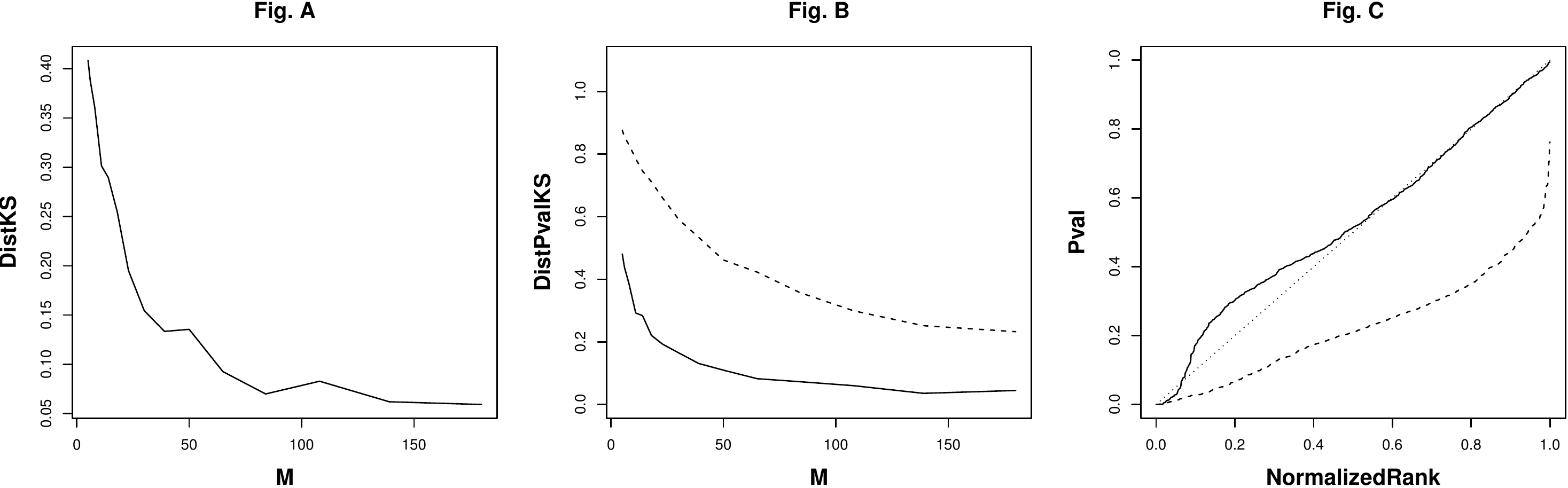}
\caption{\label{GraphPoisIndep} Under Framework {\bf $\textrm{F}_1$} (Independence assumption). {\bf Figure A.} Evolution of the Kolmogorov distance (in function of the number of trials) averaged on 1000 simulations between the empirical distribution function of the test statistics and the standard Gaussian distribution function. {\bf Figure B.} Evolution of the Kolmogorov distance averaged on 1000 simulations between the empirical distribution function of the p-values and the uniform distribution function with respect to the number of trials. The plain line stands for our test and the dashed line for the original UE one. {\bf Figure C.} Graphs of the sorted 1000 p-values (for 50 trials) in function of their normalized rank under $\mathcal{H}_0$. The plain line stands for our test, the dashed line for the original UE one and the dotted line for the uniform distribution function.}
\end{figure}

Considering $M$ independent trials of $n$ point processes, the asymptotic (with respect to $M$) of the delayed coincidence count is studied. To this aim, we use a Monte Carlo method following the procedure {\bf $\textrm{P}$} presented at the beginning of Section \ref{sec:Simulations}. On each simulation, $M$ independent trials are generated and the statistic $S_i=\sqrt{M}\left(\bar{m}_{\mathcal{L},i}-\hat{m}_{0,i}\right)/\sqrt{\hat{\sigma}_i^{2}}$ (for i from $1$ to $1000$) is computed. Theorem \ref{thm:Delta Methode} tells us that the random variables $S_i$ should be asymptotically distributed as the standard Gaussian distribution. Thus, we plot (Figure \ref{GraphPoisIndep}.A) the Kolmogorov distance $KS(F_{M,1000},F)$ between the empirical distribution function over the 1000 repetition $F_{M,1000}$ and the standard Gaussian distribution function $F$: 
\[KS(F_{M,1000},F)=\sup_x |F_{M,1000}(x)-F(x)|.\]

Usually, a test of level $\alpha=0$ always accepts, whereas a test of level $\alpha=1$ always rejects. Hence, there is a critical value (depending on the observations, and called p-value) for which the test decision passes from acceptance to rejection. If the false positive rate of a test of level $\alpha$ is exactly $\alpha$ for all $\alpha$ in $[0,1]$, which should asymptotically be the case according to Corollary \ref{cor:good:level}, then one can prove that the corresponding p-value is uniformly distributed on $[0, 1]$ under the null hypothesis. Thus, the evolution (with respect to $M$) of the Kolmogorov distance between the empirical distribution function of the obtained p-values (with our test and the one given by the UE method) and the uniform distribution function is plotted for symmetric tests (See Figure \ref{GraphPoisIndep}.B). It appears that the rate of convergence of the empirical distribution function of the p-values is faster for our test than the one given by the UE method.\\

From Figures \ref{GraphPoisIndep}.A and B, it seems reasonable to consider, for our test, sample sizes $M$ greater than $50$. Indeed, one sees that the distribution of our statistic is then almost Gaussian and the distribution of the p-values almost uniform (as expected under the null hypothesis). Thus, in order to describe more precisely what happens, we plot in Figure \ref{GraphPoisIndep}.C the sorted p-values in function of their normalized rank for $M=50$. Note that if the curve of sorted p-values is below (respectively above) the diagonal, then the observed p-values are globally smaller (respectively greater) than they should be under $\mathcal{H}_0$. Our test seems to be conservative except for big or very small p-values. The problem induced by this non conservativeness for very small p-values is detailed at the end of Section \ref{sec:Hawkes}. On the other side, the false positive rate observed for the UE test is too high. For example, we see in the figure that the UE test with a theoretical test level of $5$\% rejects almost $20$\% of the cases.\\

\subsection{Parameter Scan}
\label{sec:ParameterScan}

Here is illustrated the influence of the parameters $\lambda$ (the firing rate) and $b-a$ (the trial duration). We plot in Figure \ref{ParameterScan} the evolution (with respect to $M$) of the Kolmogorov distance between the empirical distribution function of the obtained p-values (with our test and the one given by the UE method) and the uniform distribution function.

\begin{figure}[h]
\centering{}\includegraphics[scale=0.3]{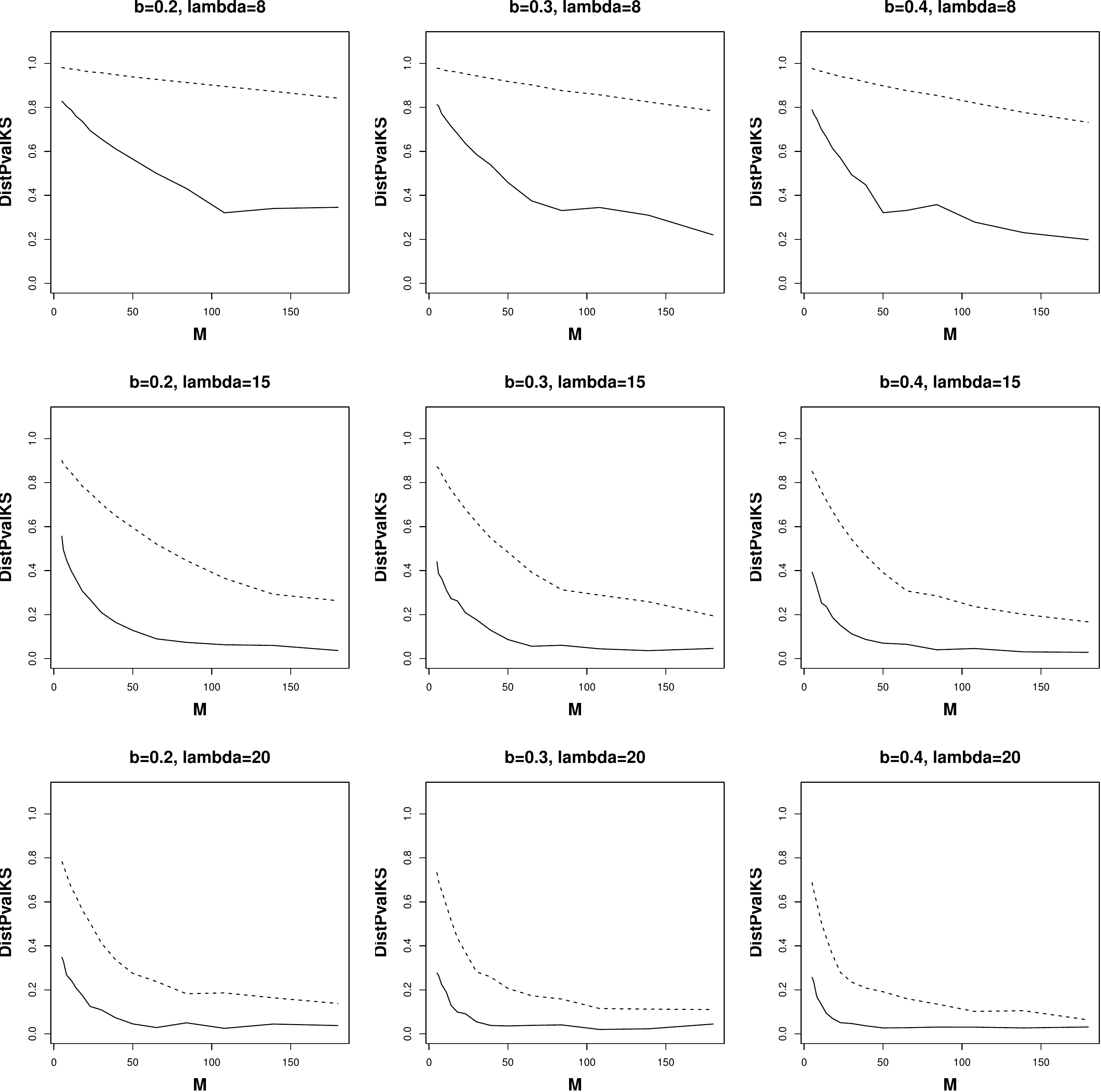}
\caption{\label{ParameterScan} Under Framework {\bf $\textrm{F}_1$} (Independence assumption). Evolution of the Kolmogorov distance averaged on 1000 simulations between the empirical distribution function of the p-values and the uniform distribution function with respect to the number of trials. The plain line stands for our test and the dashed line for the original UE one. Each plot stands for different values of $\lambda$ and $b$ (we set $a=0$ so that $b$ gives the trial duration $b-a$). From top to bottom $\lambda$ takes the values 8, 15 and 20Hz. From left to right, $b$ takes the values 0.2, 0.3 and 0.4s.}
\end{figure}

First of all, note that, if the Kolmogorov distance between the empirical distribution function of the obtained p-values and the uniform distribution function tends to $0$, then it means that the false positive rate of the test of level $\alpha$ tends to $\alpha$ when the sample size tends to infinity. As predicted by Corollary~\ref{cor:good:level}, the Kolmogorov distance between the empirical distribution function of the obtained p-values with our test tends to $0$ fast enough if $\lambda$ is not too small ($\lambda\geq 15$Hz). The test induced by the UE method seems to share the same asymptotic behaviour, but with a slower rate of convergence. Finally, it seems that our method performs better than the UE method in all the configurations of parameters.\\

In order to describe more precisely what happens, we plot in Figure \ref{ParameterScan_bis} the sorted p-values in function of their normalized rank (for $M=50$). As expected in regard of Figure \ref{ParameterScan}, the plain line sticks to the first diagonal when the parameters are large enough, since, in those cases, the KS distance between the empirical distribution function of the p-values of our test and the uniform distribution function was already small for $M=50$. However, the test induced by the UE method does not respect the prescribed level even when the parameters are large. Indeed, the dashed line remains under the diagonal in all cases. Thus, even if the asymptotic level of the UE method is good, in the practical cases where the sample size is small, the false positive rate is not guaranteed.

\begin{figure}[h]
\centering{}\includegraphics[scale=0.3]{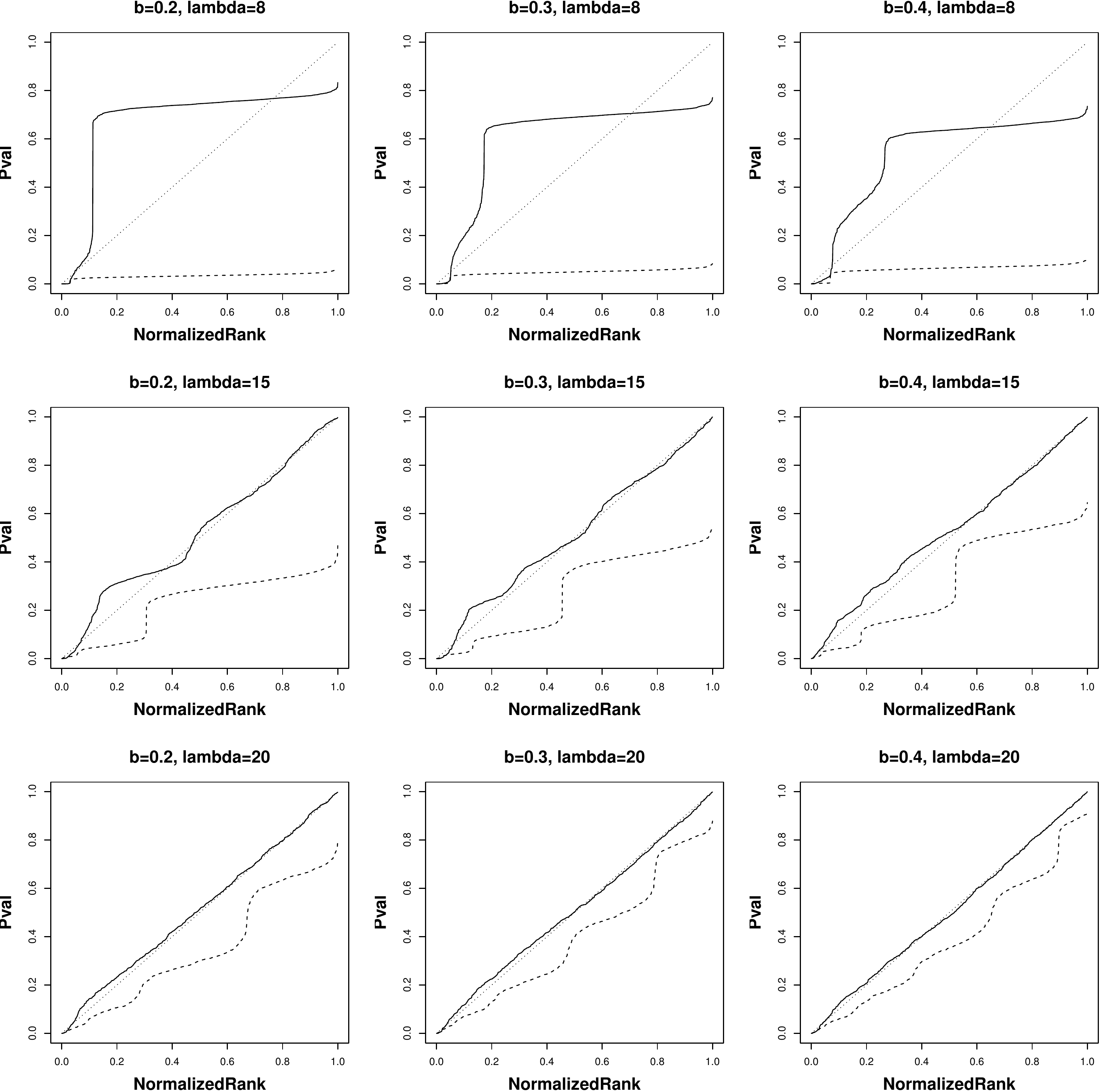}
\caption{\label{ParameterScan_bis} Under Framework {\bf $\textrm{F}_1$} (Independence assumption). Graphs of the sorted 1000 p-values (for 50 trials) in function of their normalized rank under $\mathcal{H}_0$. The plain line stands for our test, the dashed line for the original UE one and the dotted line for the uniform distribution function. Each plot stands for different values of $\lambda$ and $b$ (we set $a=0$ so that $b$ gives the trial duration $b-a$). From top to bottom $\lambda$ takes the values 8, 15 and 20Hz. From left to right, $b$ takes the values 0.2, 0.3 and 0.4s.}
\end{figure}

\clearpage

\subsection{Illustration of the true positive rate}

First, let us note that the \emph{true positive rate} of a test is the test's probability of correctly rejecting the null hypothesis. No theoretical result on this rate can be obtained from Theorem~\ref{thm:Delta Methode} who deals only with the false positive rate. So, in order to evaluate the true positive rate of the test, we simulate a sample which is dependent and check how many times the test rejects $\mathcal{H}_0$. \\

To obtain dependent Poisson processes an injection model inspired by the one used in \citep{Grun2002,Grun1999} or \citep{MTGAUE} is  used. Consider independent homogeneous Poisson processes $\overline{N}_{1},\dots,\overline{N}_{n}$, drawn according to Framework {\bf $\textrm{F}_1$}. Then, simulate an other Poisson process (according to the same framework but independent from the previous ones) $\tilde{N}$,  with an intensity of 0.3Hz, which is injected  to every neuron. Thus our sequence of dependent Poisson processes is given by 
\[ N_{i}=\overline{N}_{i} \cup{\tilde{N}}.\]
This new framework ({\bf $\textrm{F}_1$} completed by the injection) is referred as Framework {\bf $\textrm{F}_2$}.\\

Note that this injection model can only model excess of coincidences and not lack of coincidences. In the injection model used in \citep{Grun1999}, a small jitter is applied before injection to mimic temporal imprecision of the synchronous event. In our Poissonian framework this jitter cannot be performed in a similar way. Indeed, this jitter does not preserve the stationariness of the Poisson process near the edges. Although some other more elaborate injection models are available in the Poissonian framework, we do not use one of them here because their translation in the discrete time framework is not clear.

\begin{figure}[h]
\centering{}\includegraphics[scale=0.5]{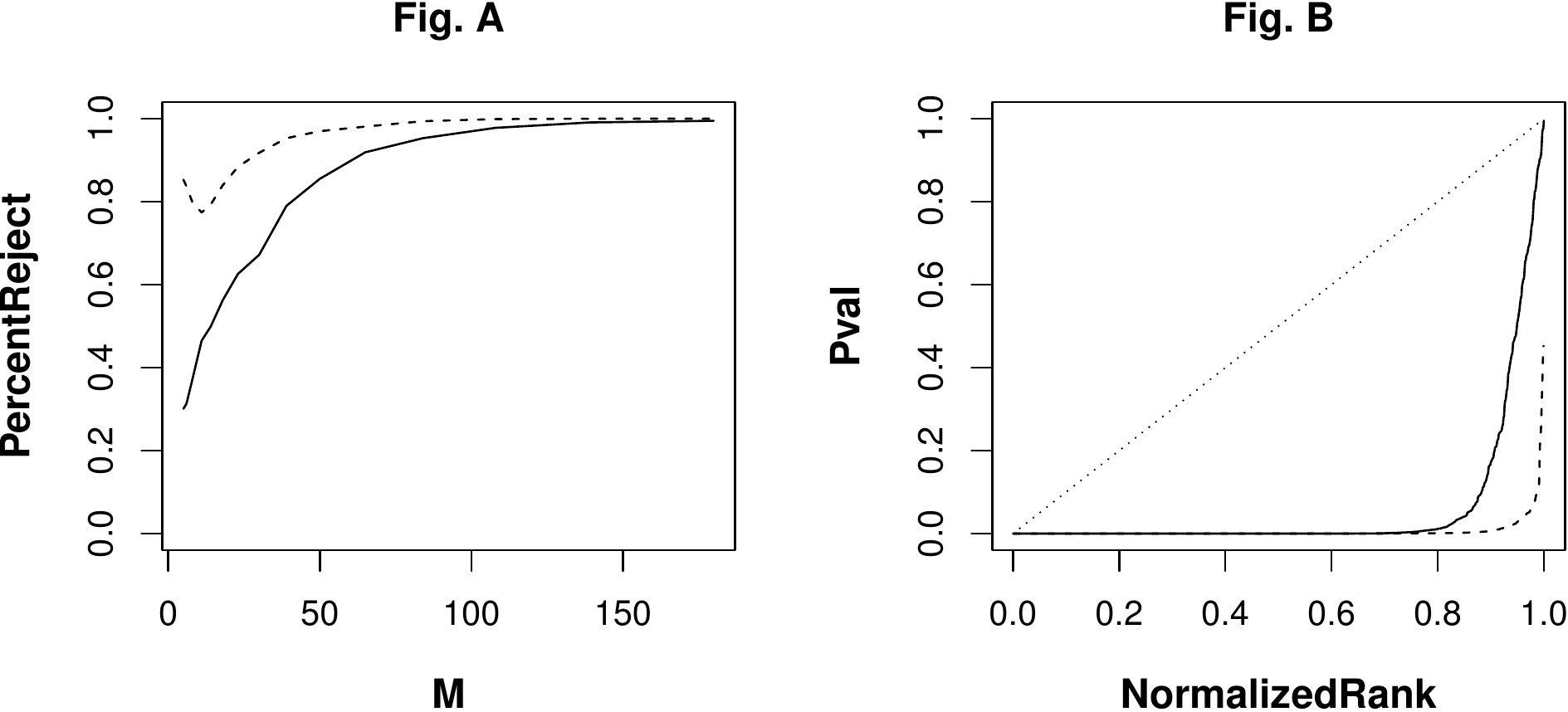}\caption{\label{GraphPoisDep} Under Framework {\bf $\textrm{F}_2$} (Dependence assumption). {\bf Figure A.} Illustration of the true positive rate of the test,  for a theoretical test level of $5$\%. The curves represent the evolution, with respect to the number of trials, of the true positive rate (averaged on 1000 simulations).   The plain line stands for our test and the dashed line for the original UE one. {\bf Figure B.} Graphs of the sorted 1000 p-values for dependent Poisson processes (50 trials). The plain line stands for our test, the dashed line for the original UE one and the dotted line for the uniform distribution function.}
\end{figure}

For a fixed theoretical level of $\alpha=5$\%, Figure \ref{GraphPoisDep}.A illustrates the true positive rate of the two tests in function of the number of trials $M$. Then Figure \ref{GraphPoisDep}.B represents the p-values as a function of their normalized rank, for $M=50$. As in Figure \ref{GraphPoisIndep}.C, the lower the curve is, the greater the observed frequency of rejection of the null hypothesis. The true positive rate is higher for the UE method for small sample sizes, but this is at the price of an  undervalued theoretical level. Indeed, we saw previously (Section \ref{sec:simu}) that the UE test gives too much false positives (for $M=50$, $20\%$ of rejection under the null hypothesis with a theoretical test level of $5\%$).

\section{Illustration Study: Non-Poissonian framework}
\label{sec:Hawkes}
In this section, a more  neurobiologically realistic framework than the Poisson one is considered. Indeed, it is interesting to see if our test is still reliable when the Poisson framework is not valid any-more. Our test is confronted to multivariate Hawkes processes, which can be simulated thanks to Ogata's thinning method \citep{Ogata81} inspired by \citep{Lewis_Simul}. The use of Hawkes processes in neurobiology was first introduced in \citep{Chornoboy1988}. With the development of simultaneous neuron recordings there is a recent trend in favour of Hawkes processes  for modelling spike trains (\citep{Pillow_nature,Krumin2010,Pernice2011,Pernice2012,MTGAUE}). Furthermore, Hawkes processes have passed some goodness-of-fit tests on real data \citep{reynaud2014goodness}. In this model, interaction between two neurons can be easily and in a more realistic way inserted. This is one of the reasons of this trend.  Note that the homogeneous Poisson process is a particular case of Hawkes processes, with no interaction between neurons.\\

A counting process $N$ is characterized by its conditional intensity $\lambda_{t}$ which is related with the local probability of finding a new point given the past. (Informally, the quantity  $\lambda_{t}dt$ gives the probability that a new point on $N$ appears in $[t, t+dt]$ given the past). The process $\left(N^{i}\right)_{i=1\ldots n}$ 
is a multivariate Hawkes process if there exist some functions $\left(h_{ij}\right)_{i,j=1\ldots n}$ (called interaction functions) and some positive constants $\left(\mu_{i}\right)_{i=1\ldots n}$ (spontaneous intensities) such that, for all $j=1,\dots ,n$, $\lambda^{j}$ given by
\[
\lambda_{t}^{j}=\max\left(0,\mu_{j}+\sum_{i=1}^n\int_{s<t}h_{ij}\left(t-s\right)\, N^{i}\left(ds\right)\right)
\]
is the intensity of the point process $N^j$, where $N^i(ds)$ is the point measure associated to $N^i$, that is \linebreak $N^i(ds)=\sum_{T\in N^i} \delta_{T}(ds)$ where $\delta_{T}$ is the Dirac measure at point $T$.\\

The functions $h_{ij}$ represent the influence of neuron $i$ over neuron $j$ in terms of spiking intensity. This influence can be either exciting ($h\geq 0$) or inhibiting ($h\leq 0$). For example, suppose that $h_{ij}=\beta\mathbf{1}_{[0,x]}$. If $\beta>0$ ({\it respectively $\beta<0$}) then the apparition of a spike on $N^i$ increases ({\it respectively decreases}) the probability to have a spike on $N^j$  during a short period of time (namely $x$): neuron $i$ excites ({\it respectively inhibits}) neuron $j$. The processes $N^i$ for $i=1,\dots ,n$ are independent if and only if $h_{ij} = 0$ for all $i \neq j$.\\ 

Note also that the self-interaction functions $h_{jj}$ can model refractory periods, making the Hawkes model more realistic than Poisson processes, even in the independence case. In particular when $h_{jj} = - \mu_j\mathbf{1}_{[0,x]}$ , all the other interaction functions being null, the $n$-dimensional process is composed by $n$ independent Poisson processes with dead time $x$, modelling strict refractory periods of length $x$ \citep{Reimer2012}.\\

All the following tests are computed according to the Framework {\bf $\textrm{F}_3$} below:

$\left. \text{\parbox{0.9\linewidth}{
\begin{itemize}
\item the trial duration of $b-a$ is randomly selected (uniform distribution) between $0.2$ and $0.4$s;
\item the $n=4$ neurons are simulated with spontaneous intensity $\mu_1,\hdots,\mu_4$ randomly selected (uniform distribution) between $8$ and $20$Hz;
\item the non-positive auto interaction functions are given by $h_{i,i}=-\mu_i\mathbf{1}_{[0,0.003s]}$;
\item the set of tested neurons is given by ${\cal L}=\left\{ 1,2,3,4\right\} $;
\end{itemize}
}} \right \}$ {\bf $\textrm{F}_3$}

 We also performed a parameter scan. However, since the results are equivalent to those obtained in the Poissonian framework, they are not presented here.\\

\subsection{Illustration of the level}

Before all, one wants to know if Theorem \ref{thm:Delta Methode} and Corollary \ref{cor:good:level} are still reliable for Hawkes processes. Thus as in section \ref{sec:Simulations}, Figure \ref{GraphHawkesIndep}.A shows the evolution of the $KS$ distance between $F_{M,1000}$ and $F$. Then as in Section \ref{sec:Simulations}, we look at the KS distance between the empirical distribution function of the p-values and the uniform distribution function to see if one can trust the level of the different tests (Figure \ref{GraphHawkesIndep}.B). These two figures are pretty similar to Figures \ref{GraphPoisIndep}.A and B (Poissonian case), but with a slightly slower convergence rate (with respect to $M$). Finally, Figure \ref{GraphHawkesIndep}.C plays the same role as Figure \ref{GraphPoisIndep}.C and presents the sorted p-values in function of their normalized rank (for $M=50$). Again, the results are comparable to those obtained in the Poissonian case: our test is rather conservative whereas the UE test rejects too many cases.

\begin{figure}[!h]
\centering{}\includegraphics[scale=0.5]{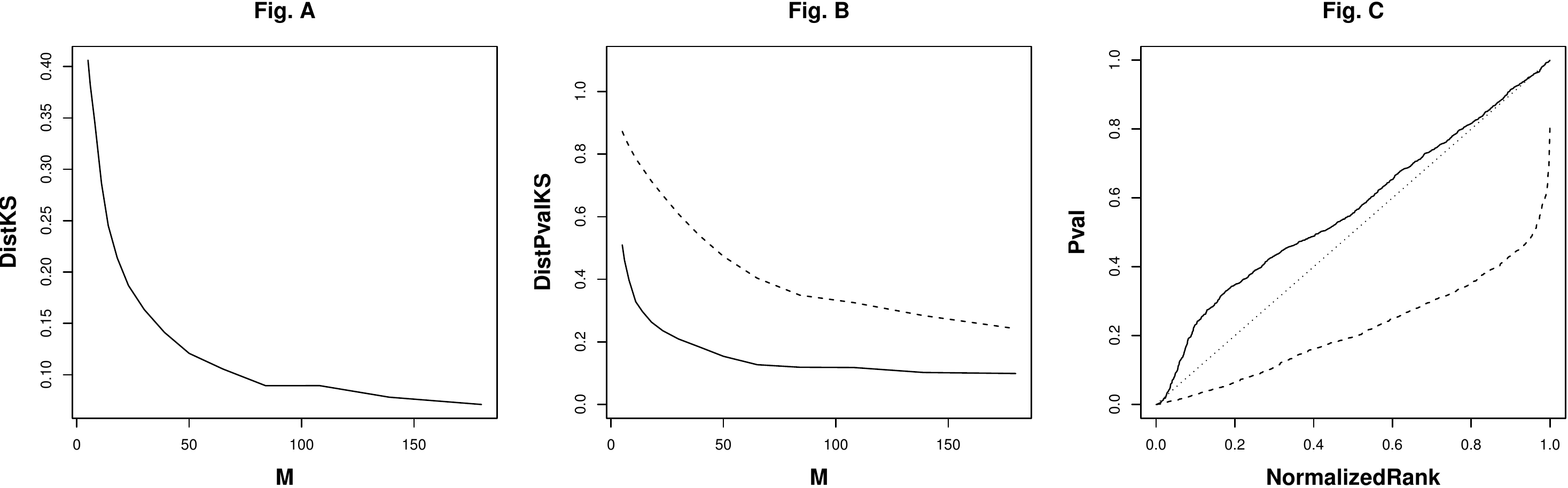}\caption{\label{GraphHawkesIndep} Under Framework {\bf $\textrm{F}_3$} (Independence assumption). {\bf Figure A.} Evolution of the Kolmogorov distance (in function of the number of trials) averaged on 1000 simulations between the empirical distribution function of the test statistics and the standard Gaussian distribution function. {\bf Figure B.} Evolution of the Kolmogorov distance averaged on 1000 simulations between the empirical distribution function of the p-values and the uniform distribution function with respect to the number of trials. The plain line stands for our test and the dashed line for the original UE one. {\bf Figure C.} Graphs of the sorted 1000 p-values (for 50 trials) in function of their normalized rank under $\mathcal{H}_0$. The plain line stands for our test, the dashed line for the original UE one and the dotted line for the uniform distribution function.}
\end{figure}

\subsection{Illustration of the true positive rate}

As said previously, it is more realistic to introduce dependency between Hawkes processes than Poisson processes. Still considering Framework {\bf $\textrm{F}_3$}, interaction functions $h_{i,j}=\beta\mathbf{1}_{[0,0.005s]}$, $\beta$ being randomly selected between 20 and 30 Hz, are added. More precisely, we add five interaction functions: $h_{1,3}$, $h_{2,3}$, $h_{1,4}$, $h_{2,4}$ and $h_{3,4}$ (summarized in Figure \ref{fig:Graphe dep loc}). Moreover, the auto interactions are updated to preserve strict refractory periods : $h_{i,i}=-(\mu_i+m_i.\beta)\mathbf{1}_{[0,0.003s]}$, where $m_i$ is the number of neurons exciting neuron $i$ (for example, $m_4=3$). This new framework ({\bf $\textrm{F}_3$} completed by the five interaction function) is referred as Framework {\bf $\textrm{F}_4$}. 

\begin{figure}[h]
\centering{}
\includegraphics[scale=0.3]{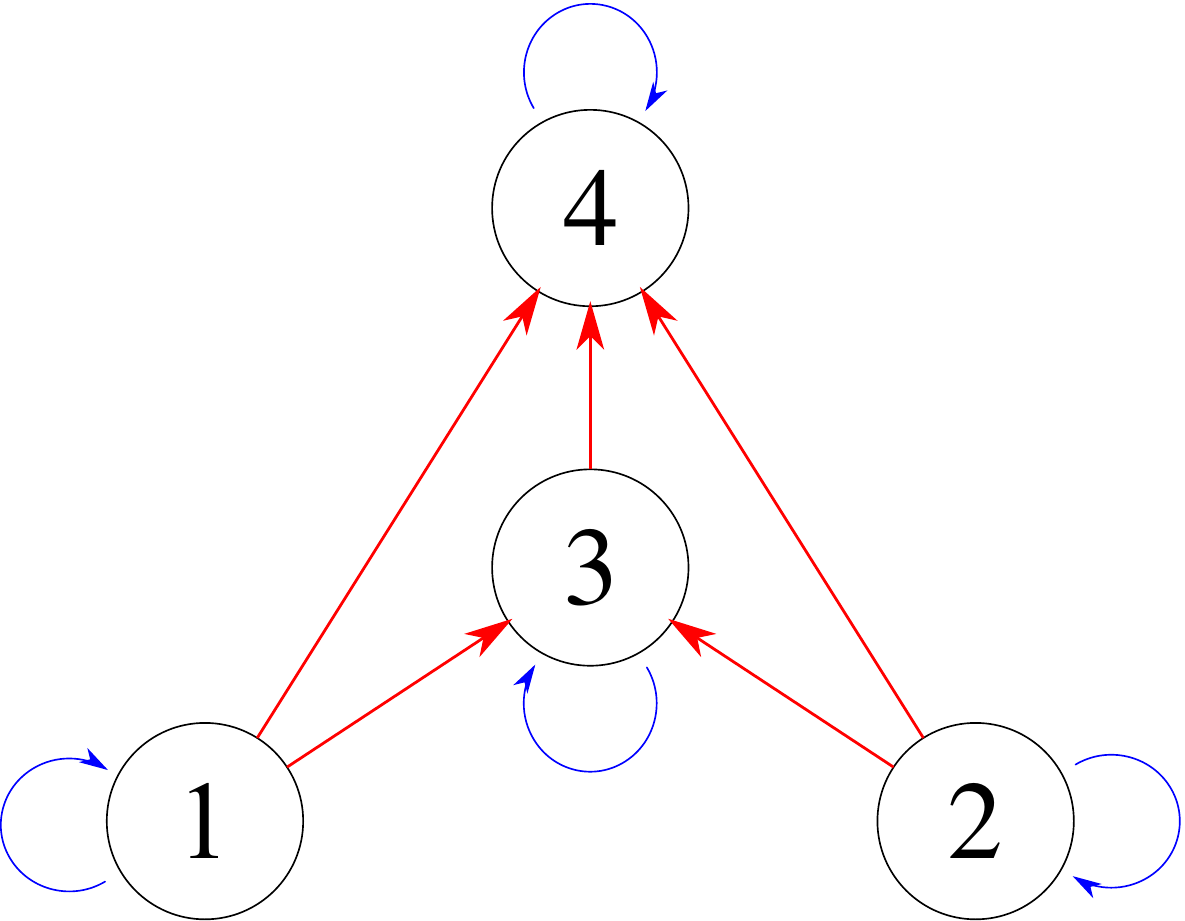}
\caption{\label{fig:Graphe dep loc}Local independence graph. An arrow means a non null interaction function. Blue arrow means inhibition and red arrow means excitation.}
\end{figure}

As previously we first provide an illustration of the true positive rate of the two tests, associated to a theoretical level of $5$\%, in function of $M$ (Figure \ref{GraphHawkesDep}.A). Then Figure \ref{GraphHawkesDep}.B represents the p-values in function of their normalized rank, for $M=50$. The difference between the true positive rates is smaller than in the Poissonian Case.

\begin{figure}[h]
\centering{}\includegraphics[scale=0.5]{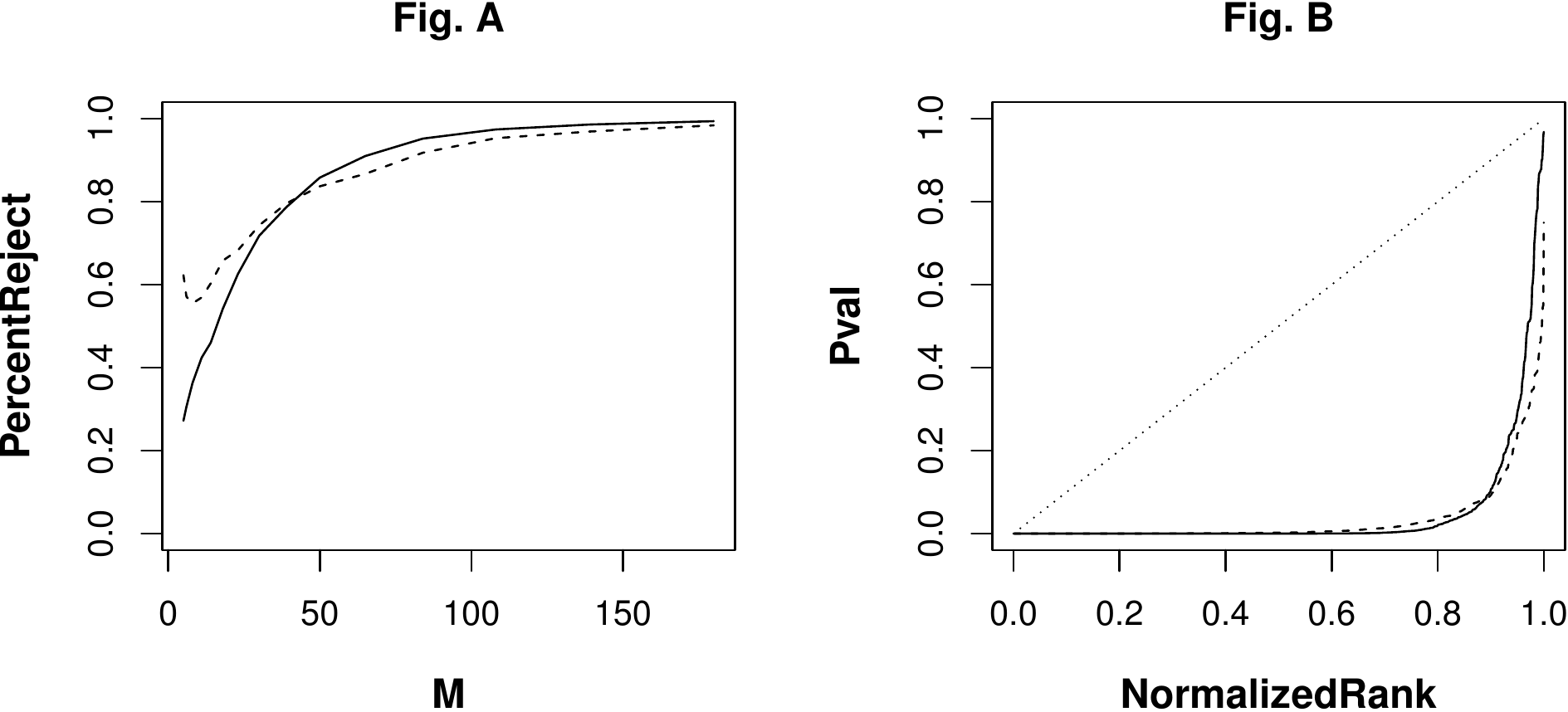}
\caption{\label{GraphHawkesDep} Under Framework {\bf $\textrm{F}_4$} (Dependence assumption, see Figure \ref{fig:Graphe dep loc}). {\bf Figure A.} Illustration of the true positive rate of the test,  for a theoretical test level of $5$\%. The curves represent the evolution, with respect to the number of trials, of the true positive rate (averaged on 1000 simulations). The plain line stands for our test and the dashed line for the original UE one. {\bf Figure B.} Graphs of the sorted 1000 p-values for 50 trials. The plain line stands for our test, the dashed line for the original UE one and the dotted line for the uniform distribution function.}
\end{figure}

\FloatBarrier

\subsection{Multiple pattern test}
\label{Mpt}
In the original MTGAUE method, a multiple testing procedure is applied with respect to $1900$ sliding time windows. In our framework, we cannot guarantee the relevance of the multiple test with this high order of multiplicity. This is due to the default of the Gaussian approximation and, more precisely, to the excess of very small p-values as noted in Section \ref{sec:simu}. But, we are able to propose a multiple testing procedure with respect to the different possible patterns. For example, with four neurons there are eleven different possible patterns, which gives a much lower order of multiplicity. So, the multiple test over all the  eleven sub-pattern of two, three or four neurons is presented here.\\

In multiple testing, the notion of false positive rate is not relevant. The closest notion might be the \emph{Family-Wise Error Rate} (FWER) which is the probability to wrongly reject at least one of the tests. This error rate can be controlled using Bonferroni's method but it is too restrictive, in particular when the number $K$ of tests involved is too large. One popular way to deal with multiple testing is the Benjamini-Hochberg procedure \citep{Benjamini1995} which ensures a control of the \emph{False Discovery Rate} (FDR). False discoveries cannot be avoided but it is not a problem if the ratio of $F_p$ (the number of false positives detections) divided by $R$ (the total number of rejects) is controlled. Therefore, the FWER and the FDR are mathematically defined by $\text{FWER}=\mathbb{P}\left(F_p>0 \right)$ and $\displaystyle \text{FDR}=\mathbb{E}\left[F_p/R \ \mathbf{1}_{R>0} \right].$\\

Note that in the full independent case, the FWER and the FDR are equal.
 The following procedure, due to Benjamini and Hochberg ensures a small FDR over $K$ tests: 

\begin{enumerate}
\item Fix a level $q$ ($q=5$\% for example);
\item Denote by $(P_1,\dots,P_{K})$ the p-values obtained for all considered tests;
\item Order them in increasing order and denote the increasing vector $(P_{(1)},\dots,P_{(K)})$;
\item Note $k_0$ the largest $k$ such that $P_{(k)}\leq k q/K$;
\item Then, reject all the tests corresponding to p-values smaller than $P_{(k_0)}$.
\end{enumerate}

The theoretical result of \citep{Benjamini1995} ensures that if the p-values are upper bounded by a uniform distribution and independently distributed under the null hypothesis, then the procedure guarantees a FDR less than $q$. The main drawback of this procedure in our case is that one needs to compute p-values that are very small when $K$ is large. For example, if $K\geq50$ and $q=5$\%, the upper bound given by $kq/K$ can be smaller than $0.001$ and as noted in Section \ref{sec:simu} the empirical frequency of very small p-values is greater than expected and therefore the uniform upper bound of the p-values is not guaranteed in our case. However, only $11$ tests are considered here and the procedure still returns reliable results.\\

We perform $1000$ simulations and count how many times each test rejects the independence. The results, obtained for $M=50$, are presented in Figure \ref{Histoboth}.

\begin{figure}[!h]
\centering{}\includegraphics[scale=0.6]{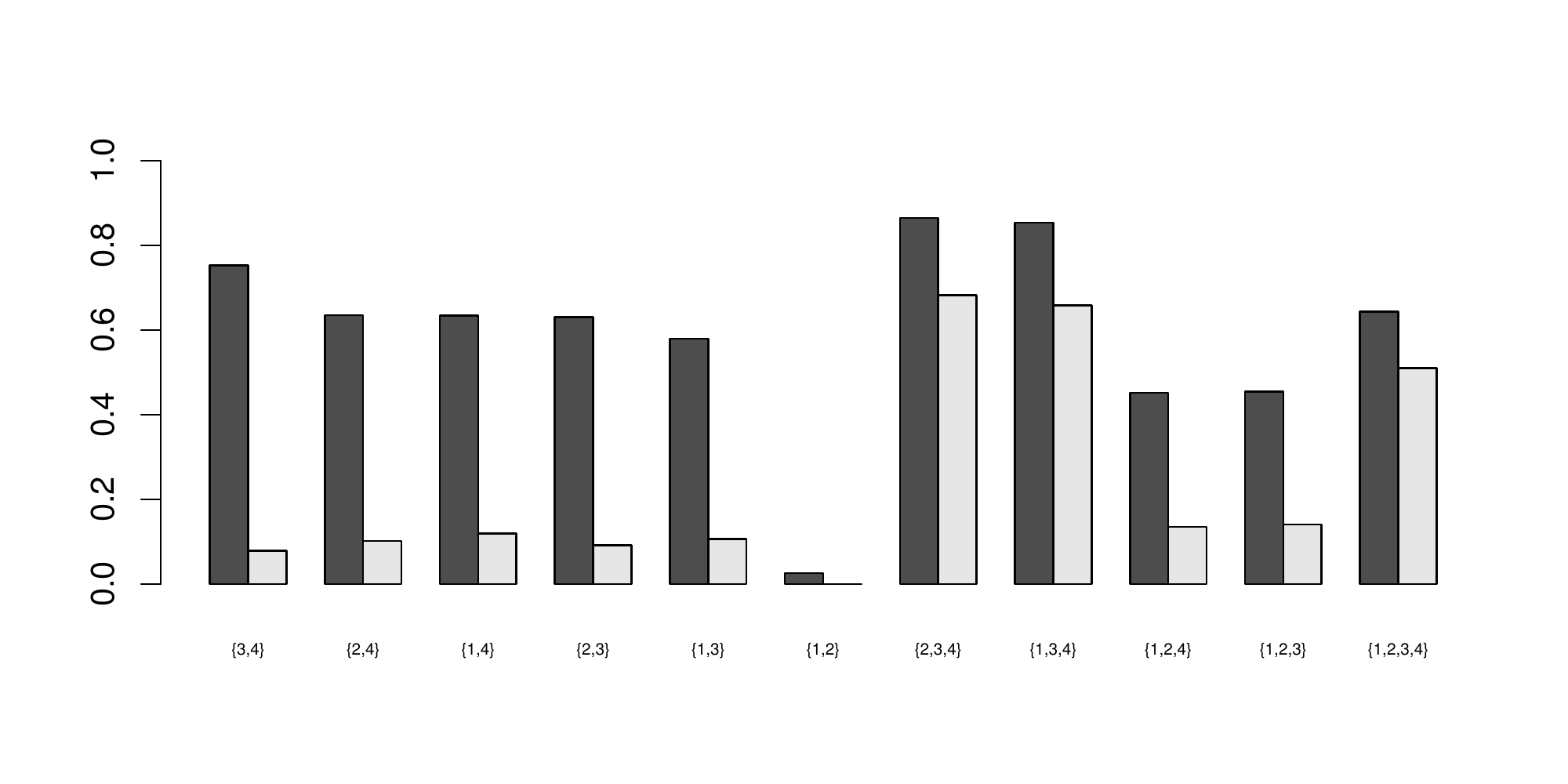}\caption{\label{Histoboth} Under Framework {\bf $\textrm{F}_4$} (Dependence assumption, see Figure \ref{fig:Graphe dep loc}). Frequency of dependence detections (1000 simulations) for each pattern. Grey for our test, white for the original UE method.}
\end{figure}

The results show that our test detects all patterns except $\{1,2\}$. This is consistent with the considered framework ({\bf $\textrm{F}_4$}) since we simulate connections between all pairs of neurons except $\{1,2\}$. The U.E. test essentially detects the patterns $\{2,3,4\}$, $\{1,3,4\}$, $\{1,2,3,4\}$ and to a lesser extent $\{1,2,4\}$ and $\{1,2,3\}$. Moreover, it misses all the pairs.

\section{Illustration on real data}
\label{sec:real}

After validating our test on simulations, we apply our method on real data and show results in agreement with classical knowledge on those data.

\subsection{Description of the data}

The data set considered here is the same  as in \citep{MTGAUE} and previous experimental studies \citep{Grammont2003,Riehle2000,Riehle2006}. The following description of the experiment is copied from Section 4.1 of \citep{MTGAUE}. These data were collected on a 5-year-old male Rhesus monkey who was trained to perform a delayed multi-directional pointing task. The animal sat in a primate chair in front of a vertical panel on which seven touch-sensitive light-emitting diodes were mounted, one in the center and six placed equidistantly (60 degrees apart) on a circle around it. The monkey had to initiate a trial by touching and then holding with the left hand the central target. After a fix delay of 500ms, the preparatory signal (PS) was presented by illuminating one of the six peripheral targets in green. After a delay of either 600ms (with probability 0.3) or 1200ms (with probability 0.7), it turned red, serving as the response signal and pointing target. Signals recorded from up to seven micro-electrodes (quartz insulated platinum-tungsten electrodes, impedance: 2-5M$\Omega$ at 1000Hz) were amplified and band-pass filtered from 300Hz to 10kHz. Using a window discriminator, spikes from only one single neuron per electrode were then isolated. Neuronal data along with behavioural events (occurrences of signals and performance of the animal) were stored on a PC for off-line analysis with a time resolution of 10kHz. The idea of the analysis is to detect some conspicuous patterns of coincident spike activity appearing during the response signal in the case of a long delay (1200ms). Therefore,  we only consider trials where the response signal is indeed occurring after a long delay.

\subsection{The test}

We have at hand the following data: spike trains associated to four neurons (35 trials by neurons). We consider two sub windows: one between 300ms and 500ms (i.e. before the preparatory signal), the other between 1100ms and 1300ms (i.e. around the expected signal). Our idea is that more synchronisation should be detected during the second window.  Moreover, we do not only want to test if the four considered neurons are independent (that is perform our test on the complete pattern $\{1,2,3,4\}$). Indeed one can be interested in knowing if neurons in some sub-patterns (for example $\{1,2\}$ or $\{1,3,4\}$ are independent. That is why we use the multiple pattern test procedure defined at the end of Section \ref{sec:Hawkes} to test all the eleven subsets (of at least two neurons) of the four considered neurons are tested. Thus we use the Benjamini-Hochberg procedure (presented in the previous section) for $K=22$ tests. Moreover, we took several values for the delay $\delta$ between 0.01s and 0.025s and the results remained stable.
\smallskip

The results are presented in Figure \ref{real}. Note that we saw in sections \ref{sec:Simulations} and \ref{sec:Hawkes}  that our test is too conservative even for small number of trials. This ensures that the theoretical level of our test can be trusted. We see that synchronizations between the subsets $\{3,4\}$ and $\{1,3,4\}$ appear in the second window. These results suggest that neurons 1, 3 and 4 belong to a neuronal assembly which is formed around the expected signal. This is in agreement with more quantitative results on those data \citep{Grammont2003,MTGAUE}. 
\vspace{1.5cm}

\begin{figure}[h]
\begin{minipage}[b][1\totalheight][t]{0.45\columnwidth}
\begin{center}
\includegraphics[scale=1]{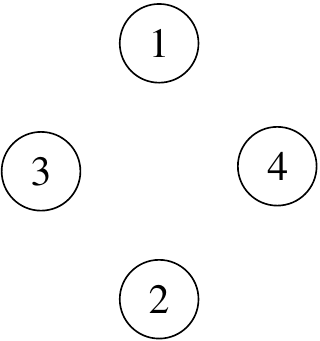}\\
300 - 500ms
\par\end{center}
\end{minipage}
\hfill{}
\begin{minipage}[b][1\totalheight][t]{0.45\columnwidth}
\begin{center}
\includegraphics[scale=1]{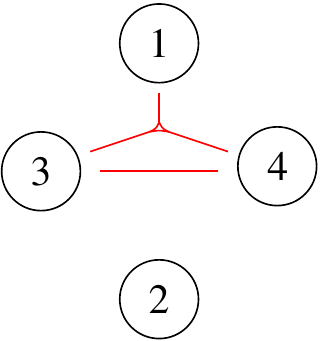}\\
1100 - 1300ms
\par\end{center}
\end{minipage}

\caption{\label{real} Evolution of the synchronization between neurons. The lines indicate the subset for which our test detects dependence. Here we detect an excess of coincidences between neurons $\{1,3,4\}$ and $\{3,4\}$}
\end{figure}

\section{Conclusion}

This paper generalizes the statistical study of the delayed coincidence count performed in \citep{MTGAUE} to more than two neurons. This delayed coincidence count  leads to an independence test for point processes which are commonly used to model spike trains.\\

Under the hypothesis that the point processes are homogeneous Poisson processes, the expectation and variance of the delayed coincidence count can be computed (Theorem \ref{thm:Moyenne et Esperance}), and then a test with prescribed  asymptotic level is built (Theorem \ref{thm:Delta Methode}). A simulation study allows us to confirm our theoretical results and to state the empirical validity of our test with a relaxed Poisson assumption. Indeed, we considered Hawkes processes which are a more realistic model of spike trains. The simulation study gives good results, even for small sample size. This allows us to use our test on real data, in order to highlight the emergence of a neuronal assembly involved at some particular time of the experiment.\\

We achieved the full generalization of the single test procedure introduced in \citep{MTGAUE}. However, we could not achieve the multiple time windows testing procedure mainly because of the default of Gaussian approximation concerning extreme values of the test statistics. More precisely, very small p-values are not distributed as expected. In particular, as noted at the end of Section \ref{sec:simu}, when the sample size $M$ is moderate ($M=50$), our test returns too many very small p-values. In \citep{MTGAUE}, the MTGAUE method is applied simultaneously on $1900$ sliding windows. In the present work,  in order to apply multiple testing both with respect to the sliding time windows and the subsets, the total number of tests is even larger. Indeed, for each sliding window, there are $2^n-n-1$ tests to perform, where $n$ is the number of recorded neurons. As said at the end of Section \ref{sec:Hawkes} this would lead to extremely small p-values, for which our test is less reliable.\\

Even if our test remains empirically reliable under a non Poissonian framework, it could be therefore of interest to explore surrogate data method such as trial-shuffling \citep{Pipa2003}. A very recent work based on permutation approach for delayed coincidence count with $n=2$ neurons \citep{Albert2014} is a first step in this direction but needs to be generalized to more than $2$ neurons.\\

\textbf{Acknowledgements}

We first of all want to thank Alexa Riehle, leader of the laboratory in which the data
used in this article were previously collected, and Franck Grammont who collected these data. 
Finally we thank Patricia Reynaud-Bouret for fruitful discussions.

\clearpage

\section{Proofs}
\label{app: expectation and variance}
As said in Section \ref{sub:Asymptotic:properties}, we prove more general results than Theorems \ref{thm:Moyenne et Esperance} and \ref{thm:Delta Methode}. Considering $N_{1},\dots,N_{n}$, some point processes on $\left[a,b\right]$ and $\mathcal{L}\subset \{1,\dots ,n\}$ a set of indices with cardinal $L\geq 2$, we prove the same kind of results with any coincidence function $c\left(x_{1},\ldots,x_{L}\right)$ with value either $0$ or $1$ satisfying Definition \ref{def: coincidence} below.

\begin{defin}
\label{def: coincidence}
$\left.\right.$

\begin{enumerate}
\item A coincidence function is a function $c:\left[a,b\right]^{L}\rightarrow\left\{ 0,1\right\} $ which is symmetric.

\item Let $\left(x_{1},\ldots,x_{L}\right)\in \prod_{l\in\mathcal{L}} N_{l}$
be a $L$-tuple with a spiking time of every neuron of the subset $\mathcal{L}$. Say that $\left(x_{1},\ldots,x_{L}\right)$ is a coincidence if and only if $c\left(x_{1},\ldots,x_{L}\right)=1$.

\item Given $c$ a coincidence function, we define $X_{\mathcal{L}}$ the number of coincidences on $[a,b]$ by:
\[
X_{\mathcal{L}} = \sum_{(x_1,\dots,x_L)\in \prod_{l\in \mathcal{L}} N_{l}}
c\left(x_{1},\ldots,x_{L}\right).
\]
\item Define
\[
\forall k\in\{0,\dots,L\},\ 
I(L,k)=\intop_{[a,b]^{L-k}}
\left(\intop_{[a,b]^{k}}c\left(x_{1},\ldots,x_{L}\right)\,
dx_{1}\ldots dx_{k}\right)^{2}dx_{k+1}\ldots dx_{L}
\]
where the convention $\intop\limits _{[a,b]^{0}}f\left(x\right) dx=f\left(x\right)$ is set.
\end{enumerate}

\end{defin}

\subsection{Proof of Theorem~\ref{thm:Moyenne et Esperance}\\}

\begin{thm}
\label{thm:Moyenne et Esperance general}Under assumptions and notations of Definition \ref{def: coincidence}, if $N_{1},\dots,N_{n}$ are some independent homogeneous Poisson processes on $[a,b]$ with intensities $\lambda_{1},\dots,\lambda_{n}$, the expected value and the variance of the number of coincidences $X_{\mathcal{L}}$ are given by:
\begin{equation}\label{eq:moyenne:nb:coin}
m_{0,\mathcal{L}}:=\mathbb{E}\left[X_{\mathcal{L}}\right]=\left(\prod_{l\in\mathcal{L}}
\lambda_{l}\right)I(L,0)
\end{equation}
and
\begin{equation}\label{eq:variance:nb:coin}
\Var(X_{\mathcal{L}})=m_{0,\mathcal{L}}+\sum_{k=1}^{L-1}\left(\sum_{\begin{subarray}{c}
\mathcal{J}\subset\mathcal{L}\\
\#\mathcal{J}=k
\end{subarray}}\prod_{j\in\mathcal{J}}\lambda_{j}^{2}
\prod_{l\notin\mathcal{J}}\lambda_{l}\right)I(L,k).
\end{equation}

\end{thm}

\begin{proof}
By definition,
\[
\mathbb{E}\left[X_{\mathcal{L}}\right] = \mathbb{E}\left[\sum_{(x_1,\dots,x_L)\in \prod_{l\in \mathcal{L}} N_{l}} c\left(x_{1},\ldots,x_{L}\right)\right].
\]
Using the fact that $N_{1},\cdots,N_{n}$ are independent homogeneous
Poisson processes with respective intensities $\lambda_{1},\cdots,\lambda_{n}$
one can prove (see \citep{Daley_PP_1}) that 
\[
\mathbb{E}\left[X_{\mathcal{L}}\right]=\left(\prod_{l\in \mathcal{L}}\lambda_{l}\right)\int\limits _{[a,b]^{L}}c\left(x_{1},\ldots,x_{L}\right)\ dx_{1}\ldots dx_{L} = \left(\prod_{l\in\mathcal{L}} \lambda_{l}\right)I(L,0).
\]
For sake of simplicity, the variance is computed in the simpler case where $\mathcal{L}=\{1,...,L\}$, the generalization being pretty clear. In order to simplify, we use the integral form of the coincidence count, i.e.
\begin{equation*}
X_{\mathcal{L}} = \int\limits _{[a,b]^{L}}c\left(x_{1},\ldots,x_{L}\right)\ dN_{1}(x_{1})\dots dN_{L}(x_{L})
\end{equation*}
where $dN_{1},\dots,dN_{L}$ are the point measures associated to $N_{1},\dots,N_{L}$. Thanks to Fubini Theorem we have
\begin{equation}\label{eq:esperance:X:2}
\mathbb{E}\left[X_{\mathcal{L}}^{2}\right]=\mathbb{E}\left[\int\limits _{[a,b]^{2L}}c\left(x_{1},\ldots,x_{L}\right)c\left(y_{1},\ldots,y_{L}\right)\ \prod_{l=1}^{L}dN_{l}\left(x_{l}\right)dN_{l}\left(y_{l}\right)\right].
\end{equation}
Then, let us define 
\begin{equation}\label{eq:def:a:b:1:a:b:2}
\left[a,b\right]^{\left(1\right)}=\left\{ \left(x,y\right)\in[a,b]^{2}\ |\ x=y\right\}
\text{ and } \left[a,b\right]^{\left(2\right)}=\left[a,b\right]^{2}\setminus\left[a,b\right]^{\left(1\right)}.
\end{equation}
Now, let us see that $[a,b]^{2L}=\left(\left[a,b\right]^{2}\right)^{L} = \sqcup_{\varepsilon\in \{1,2\}^{L}} \left( \prod_{l=1}^{L} [a,b]^{(\varepsilon_{l})} \right)$ where $\varepsilon_{l}$ denotes the $l$-th coordinate of $\varepsilon$. Using this decomposition and Equation~\eqref{eq:esperance:X:2}, it is clear that
\begin{equation}\label{eq:esperance:X:2:decomp}
\mathbb{E}\left[X_{\mathcal{L}}^{2}\right]=\sum_{\varepsilon\in \{1,2\}^{L}} A_{\varepsilon},
\end{equation}
where for all $\varepsilon$ in $\{1,2\}^{L}$,
\begin{equation*}
A_{\varepsilon}=\mathbb{E}\left[\ \int\limits _{\prod\limits_{l=1}^{L} [a,b]^{(\varepsilon_{l})}}c\left(x_{1},\ldots,x_{L}\right)c\left(y_{1},\ldots,y_{L}\right)\ \prod_{l=1}^{L}dN_{l}\left(x_{l}\right)dN_{l}\left(y_{l}\right)\right].
\end{equation*}
For every $p=1,\dots,L$, let $\varepsilon^{(p)}=(1,\dots,1,2,\dots,2)$ where the number of $1$'s in $\varepsilon^{(p)}$ is exactly $p$.  Properties of the moment measure of Poisson processes (see \citep{Daley_PP_1} or \citep{Poisson_Process} in a more simplified framework) lead to
\begin{multline}
\label{eq:Arg}
A_{\varepsilon^{(p)}}=\mathbb{E}\left[ \int\limits _{\left(\left[a,b\right]^{\left(1\right)}\right)^{p}}\int\limits _{\left(\left[a,b\right]^{\left(2\right)}\right)^{L-p}}c\left(x_{1},\ldots,x_{L}\right)c\left(y_{1},\ldots,y_{L}\right)\ \prod_{l=1}^{L}dN_{l}\left(x_{l}\right)dN_{l}
\left(y_{l}\right)\right]\\
\quad \quad = \prod_{l=1}^{p}\lambda_{l}\prod_{j=p+1}^{L}\lambda_{j}^{2}\intop_{\left[a,b\right]^{p}}\left(\intop_{\left[a,b\right]^{2\left(L-p\right)}}c\left(t_{1},\ldots,t_{p},x_{p+1},\ldots,x_{L}\right) \right.\\
\left. c\left(t_{1},\ldots,t_{p},y_{p+1},\ldots,y_{L}\right)\,\prod_{k=p+1}^{L}dx_{k}dy_{k}\right)dt_{1}\ldots dt_{p}.
\end{multline}
For fixed $(t_{1},\ldots,t_{p})$ one can apply Fubini Theorem to the
inner integral which leads to:
\begin{eqnarray*}
A_{\varepsilon^{(p)}} &= &\prod_{l=1}^{p}\lambda_{l}\prod_{j=p+1}^{L}\lambda_{j}^{2}\intop_{\left[a,b\right]^{p}}\left(\intop_{\left[a,b\right]^{2\left(L-p\right)}}c\left(t_{1},\ldots,t_{p},t_{p+1},\ldots,t_{J}\right)dt_{p+1}\dots dt_{L}\right)^{2}dt_{1}\ldots dt_{p}.\\
& = &\prod\limits _{l=1}^{p}\lambda_{l}\prod\limits _{j=p+1}^{L}\lambda_{j}^{2}\ I\left(L,L-p\right).
\end{eqnarray*}
by definition of $I(L,L-p)$. 

For more general vectors $\varepsilon$ in $\{1,2\}^{L}$, let us note $p$ the occurrence count of $1$ in the vector $\varepsilon$ and $I_{\varepsilon}$ (respectively $J_{\varepsilon}$) the set of indices of the coordinates of $\varepsilon$ equal to $1$ (respectively $2$). Then, using the symmetry of the coincidence function $c$, one can easily deduce from the computation of $A_{\varepsilon^{(p)}}$ that
\begin{equation}\label{eq:A:epsilon}
A_{\varepsilon} = \prod\limits _{i\in I_{\varepsilon}}\lambda_{i}\prod\limits _{j\in J_{\varepsilon}}\lambda_{j}^{2}\ I\left(L,L-p\right).
\end{equation}
From~\eqref{eq:esperance:X:2:decomp} and~\eqref{eq:A:epsilon}, one deduces
\begin{equation*}
\mathbb{E}\left[X_{\mathcal{L}}^{2}\right]= \sum_{p=0}^{L}\left(\sum_{
\begin{subarray}{c}
\mathcal{J}\subset\mathcal{L}\\
\#\mathcal{J}=p
\end{subarray}}
\prod_{j\in\mathcal{J}}\lambda_{j}\prod_{l\notin\mathcal{J}}\lambda_{l}^{2}\right)I(L,L-p).
\end{equation*}
Note that $I(L,L)=I(L,0)^{2}$ by definition, so the case $p=0$ in the sum corresponds to 
$$\prod\limits _{l\in \mathcal{L}}\lambda_{l}^{2}\ I\left(L,L\right)= \prod\limits _{l\in \mathcal{L}}\lambda_{l}^{2}\ I\left(L,0\right)^{2} = \mathbb{E}\left[ X_{\mathcal{L}} \right]^{2}.$$
Moreover, the case $p=L$ in the sum corresponds to 
$\prod\limits _{l\in \mathcal{L}}\lambda_{l}\ I\left(L,0\right)= \mathbb{E}\left[ X_{\mathcal{L}} \right]$.
So, we have
\begin{equation*}
\mathbb{E}\left[X_{\mathcal{L}}^{2}\right]= \mathbb{E}\left[ X_{\mathcal{L}} \right]^{2} + \mathbb{E}\left[ X_{\mathcal{L}} \right] + \sum_{p=1}^{L-1}\left(\sum_{
\begin{subarray}{c}
\mathcal{J}\subset\mathcal{L}\\
\#\mathcal{J}=p
\end{subarray}}
\prod_{j\in\mathcal{J}}\lambda_{j}\prod_{l\notin\mathcal{J}}\lambda_{l}^{2}\right)I(L,L-p),
\end{equation*}
and~\eqref{eq:variance:nb:coin} clearly follows by defining the variable $k=L-p$.

\end{proof}

Theorem \ref{thm:Moyenne et Esperance} is a direct consequence of Theorem \ref{thm:Moyenne et Esperance general} since the function $c_\delta : \left[a,b\right]^{L}\rightarrow\left\{ 0,1\right\}$ defined by
\begin{equation}\label{eq:def:fun:coin}
c_\delta(x_1,\hdots,x_n)=\mathbf{1}_{\left| \max\limits_{i\in\{1,\dots,L\}}x_{i} -\min\limits_{i\in\{1,\dots,L\}}x_{i} \right|\leq\delta},\; 0< \delta <\frac{b-a}{2}
\end{equation}

satisfies Definition \ref{def: coincidence}.

\subsection{Proof of Theorem~\ref{thm:Delta Methode}}
\label{app: delta method}

\begin{thm}
\label{thm:Delta Methode general}Under Notations and Assumptions of Theorem \ref{thm:Moyenne et Esperance general},  the two following affirmations are valid:

\begin{itemize}
\item The following convergence of distribution holds:
\[
\sqrt{M}\left(\bar{m}_{\mathcal{L}}-\hat{m}_{0,\mathcal{L}}\right)\underset{M\rightarrow\infty}{\overset{\mathcal{D} }{\longrightarrow}}\mathcal{N}\left(0,\sigma^{2}\right),
\]
where
\[
\sigma^{2}=\Var(X_{\mathcal{L}})-(b-a)^{-1}\mathbb{E}\left[X_{\mathcal{L}}\right]^{2}\left(\sum_{l\in \mathcal{L}} \lambda_{l}^{-1}\right).
\]

\item Moreover, $\sigma^{2}$ can be estimated by 
\[
\hat{\sigma}^{2}=\hat{v}\left(X_{\mathcal{L}}\right)-(b-a)^{-1}I(L,L)\prod_{l\in \mathcal{L}}\hat{\lambda}_{l}^{2}\left(\sum_{k\in \mathcal{L}}\hat{\lambda}_{k}^{-1}\right),
\]
where
\[
\hat{v}(X_{\mathcal{L}})=\hat{m}_{0,\mathcal{L}} +\sum_{k=1}^{L-1}
\left(
\sum_{\begin{subarray}{c}
\mathcal{J}\subset\mathcal{L}\\
\#\mathcal{J}=k
\end{subarray}}
\prod_{j\in\mathcal{J}}\hat{\lambda}_{j}^{2} \prod_{l\notin\mathcal{J}}\hat{\lambda}_{l}
\right)
I(L,k),
\]
and
\[
\sqrt{M}\frac{\left(\bar{m}_{\mathcal{L}}-\hat{m}_{0,\mathcal{L}}\right)}{\sqrt{\hat{\sigma}^{2}}}\overset{\mathcal{D}}{\rightarrow}\mathcal{N}\left(0,1\right).
\]
\end{itemize}

\end{thm}

\begin{proof}
For sake of simplicity, the result is proven in the simpler case where $\mathcal{L}=\{1,...,L\}$, the generalization being pretty clear.
An application of the Central Limit Theorem leads to:
\[
\frac{1}{\sqrt{M}}\sum_{k=1}^M \left[\left(\begin{array}{c}
X_{\mathcal{L}}^{\left(k\right)}\\
N_{1}^{\left(k\right)}\left(\left[a,b\right]\right)\\
\vdots\\
N_{L}^{\left(k\right)}\left(\left[a,b\right]\right)
\end{array}\right)-\left(\begin{array}{c}
\mathbb{E}\left[X_{\mathcal{L}}\right]\\
\lambda_{1}(b-a)\\
\vdots\\
\lambda_{L}(b-a)
\end{array}\right)\right]\overset{\mathcal{D}}{\rightarrow}\mathcal{N}_{L+1}\left({\bf 0},\Gamma\right),
\]
where $\mathcal{N}_{L+1}\left({\bf 0},\Gamma\right)$ is the multivariate Gaussian distribution with $L+1$-dimensional mean vector ${\bf 0}$ and covariance matrix $\Gamma$ defined by: 
\[
\Gamma=\left(\begin{array}{cccc}
\Var(X_{\mathcal{L}}) & \mathbb{E}\left[X_{\mathcal{L}}\right] & \cdots & \mathbb{E}\left[X_{\mathcal{L}}\right]\\
\mathbb{E}\left[X_{\mathcal{L}}\right] & \lambda_{1}(b-a) & 0 & 0\\
\vdots & 0 & \ddots & 0\\
\mathbb{E}\left[X_{\mathcal{L}}\right] & 0 & 0 & \lambda_{L}(b-a)
\end{array}\right)
\]
The matrix is obtained using the fact that the processes $N_{l}$, $l\in \mathcal{L}$
are independent and from the following computation for all $j_{0}$ in $\mathcal{L}$,
\begin{eqnarray*}
\mathbb{E}\left[X_{\mathcal{L}}N_{j_{0}}\left(\left[a,b\right]\right)\right] & = & \mathbb{E}\left[\int\limits _{[a,b]^{L+1}}c\left(x_{1},\ldots,x_{L}\right)\ dN_{1}\left(x_{1}\right)\ldots dN_{L}\left(x_{L}\right)dN_{j_{0}}\left(y\right)\right]\\
 & = & \mathbb{E}\left[\int\limits _{[a,b]^{L-1}}\left(\intop\limits _{\left[a,b\right]^{\left(2\right)}}c\left(x_{1},...,x_{L}\right)\ dN_{j_{0}}\left(x_{j_{0}}\right)dN_{j_{0}}\left(y\right)\right) \prod_{l\in \mathcal{L}, l\neq j_{0}}dN_{j}\left(x_{j}\right)\right]\\
 &  & +\mathbb{E}\left[\int\limits _{[a,b]^{L-1}}\left(\intop\limits _{\left[a,b\right]^{\left(1\right)}}c\left(x_{1},...,x_{L}\right)\ dN_{j_{0}}\left(x_{j_{0}}\right)dN_{j_{0}}\left(y\right)\right) \prod_{l\in \mathcal{L}, l\neq j_{0}}dN_{j}\left(x_{j}\right)\right]\\
 & = & \lambda_{j_{0}} \left(\prod_{l\in \mathcal{L}} \lambda_{l} \right) \int\limits _{[a,b]} \left( \int\limits _{[a,b]^{L}}c\left(x_{1},\ldots,x_{L}\right)\ dx_{1}\ldots dx_{L}\right)dy+\left(\prod_{l\in \mathcal{L}} \lambda_{l} \right)I(L,0)\\
&&  \textrm{(using the same arguments as for \eqref{eq:Arg})}\\
 & = & \lambda_{j_{0}}(b-a)\left(\prod_{l\in \mathcal{L}} \lambda_{l} \right)I(L,0)+\left(\prod_{l\in \mathcal{L}} \lambda_{l} \right)I(L,0)\\
 & = & \mathbb{E}\left[N_{j_{0}}\left(\left[a,b\right]\right)\right]\mathbb{E}\left[X_{\mathcal{L}}\right]+\mathbb{E}\left[X_{\mathcal{L}}\right],
\end{eqnarray*}
where $\left[a,b\right]^{\left(2\right)}$ and $\left[a,b\right]^{\left(1\right)}$ are defined by~\eqref{eq:def:a:b:1:a:b:2} in the previous proof.
Define 
$$g:(x,u_{1},\ldots,u_{L})\mapsto x- (b-a)^{-L}I(L,0)\prod_{l=1}^{L}u_{l},$$
and remark that:
$$
\left\{
\begin{aligned}
&g\left(\frac{1}{M}\sum_{k=1}^M X_{\mathcal{L}}^{\left(k\right)},\frac{1}{M}\sum_{k=1}^M N_{1}^{\left(k\right)}\left(\left[a,b\right]\right),\ldots,\frac{1}{M}\sum_{k=1}^M N_{L}^{\left(k\right)}\left(\left[a,b\right]\right)\right)=\bar{m}_{\mathcal{L}}-\hat{m}_{0,\mathcal{L}}, \\
&g\left(\mathbb{E}\left[X_{\mathcal{L}}\right],\lambda_{1}(b-a),\ldots,\lambda_{L}(b-a)\right)=0 \quad\textrm{ (thanks to Theorem \ref{thm:Moyenne et Esperance general})}.
\end{aligned}
\right.
$$
So we have
\begin{eqnarray*}
\sqrt{M}\left(\bar{m}_{\mathcal{L}}-\hat{m}_{0,\mathcal{L}}\right) & = & \sqrt{M}\left[g\left(\frac{1}{M}\sum_{k=1}^M X_{\mathcal{L}}^{\left(k\right)},\frac{1}{M}\sum_{k=1}^M N_{1}^{\left(k\right)}\left(\left[a,b\right]\right),\ldots,\frac{1}{M}\sum_{k=1}^M N_{L}^{\left(k\right)}\left(\left[a,b\right]\right)\right)\right.\\
 &  & \left.-g\left(\mathbb{E}\left[X_{\mathcal{L}}\right],\lambda_{1}(b-a),\ldots,\lambda_{L}(b-a)\right)\vphantom{\sum_{k=1}^M}\right].
\end{eqnarray*}
And the delta method \citep{Casella2002} gives the following convergence of distribution,
\[
\sqrt{M}\left(\bar{m}_{\mathcal{L}}-\hat{m}_{0,\mathcal{L}}\right)\underset{}{\overset{\mathcal{D}}{\longrightarrow}}\mathcal{N}\left(0,\vphantom{D}^{t}\negmedspace D\Gamma D\right),
\]
where $D$ is the gradient of the function $g$ at the point $\left(\mathbb{E}\left[X_{\mathcal{L}}\right],\lambda_{1}(b-a),\ldots,\lambda_{L}(b-a)\right)$
i.e.
\[
D=\left(\begin{array}{c}
1\\
-\lambda_{1}^{-1}\mathbb{E}\left[X_{\mathcal{L}}\right](b-a)^{-1}\\
\vdots\\
-\lambda_{L}^{-1}\mathbb{E}\left[X_{\mathcal{L}}\right](b-a)^{-1}
\end{array}\right).
\]
So,
\begin{eqnarray*}
\vphantom{D}^{t}\negmedspace D\Gamma D & = &\vphantom{D}^{t}\negmedspace D
\left(\begin{array}{c}
\Var(X_{\mathcal{L}})-(b-a)^{-1}
\mathbb{E}\left[X_{\mathcal{L}}\right]^{2}
\left(\sum_{l\in\mathcal{L}}\lambda_{l}^{-1}\right)\\
\mathbb{E}\left[X_{\mathcal{L}}\right]-\mathbb{E}\left[X_{\mathcal{L}}\right]\\
\vdots\\
\mathbb{E}\left[X_{\mathcal{L}}\right]-\mathbb{E}\left[X_{\mathcal{L}}\right]
\end{array}\right)\\
& = & \Var(X_{\mathcal{L}})-(b-a)^{-1} \mathbb{E}\left[X_{\mathcal{L}}\right]^{2} \left(\sum_{l\in\mathcal{L}}\lambda_{l}^{-1}\right),
\end{eqnarray*}
which proves the first part of the Theorem \ref{thm:Delta Methode general}.

To get the second part, it suffices to apply Slutsky lemma \citep{Casella2002} since the $\hat{\lambda}_l$'s are consistent.
\end{proof}

Once again, Theorem \ref{thm:Delta Methode} is a direct consequence of Theorem \ref{thm:Delta Methode general} since the function $c_\delta : \left[a,b\right]^{L}\rightarrow\left\{ 0,1\right\}$ defined by~\eqref{eq:def:fun:coin} satisfies Definition \ref{def: coincidence}.

\subsection{Proof of Proposition \ref{prop:Formule des I(k)}}
Here we compute
\[
I\left(L,k\right)=\intop_{[a,b]^{L-k}}\left(\intop_{[a,b]^{k}}\mathbf{1}_{|\max\left(\vee x_{i},\vee y_{i}\right)-\min\left(\wedge x_{i},\wedge y_{i}\right)|\leq\delta}\, dx_{1}\ldots dx_{k}\right)^{2}dy_{1}\ldots dy_{L-k}
\]
where $\wedge x_{i}=\min\left\{ x_{i},i\in\left\{ 1,\dots,k\right\} \right\}$, $\vee x_{i}=\max\left\{ x_{i},i\in\left\{ 1,\dots,k\right\} \right\}$ and respectively for $\wedge y_{i}$ and $\vee y_{i}$. Let us fix some $(y_{1},\dots ,y_{L-k})$ in $[a,b]^{L-k}$ and compute the inner integral
\[
\Sigma\left( y_{1},\dots ,y_{L-k} \right)=\intop_{\left[a,b\right]^{k}}\mathbf{1}_{|\max\left(\vee x_{i},\vee y_{i}\right)-\min\left(\wedge x_{i},\wedge y_{i}\right)|\leq\delta}\ dx_{1}\ldots dx_{k}.
\]
In order to do that let us decompose the integral with respect to the following conditions on $(x_{1},\dots ,x_{k})$:
\begin{enumerate}
\item if $\wedge x_{i}>\wedge y_{i}$ and $\vee x_{i}>\vee y_{i}$, denote
the integral $A$;
\item if $\wedge x_{i}<\wedge y_{i}$ and $\vee x_{i}<\vee y_{i}$, denote
the integral $B$;
\item if $\wedge x_{i}>\wedge y_{i}$ and $\vee x_{i}<\vee y_{i}$, denote
the integral $C$;
\item if $\wedge x_{i}<\wedge y_{i}$ and $\vee x_{i}>\vee y_{i}$, denote
the integral $D$.
\end{enumerate}
Since we have partitioned $\left[a,b\right]^{k}$ up to a null measure
set, we have $\Sigma\left( y_{1},\dots ,y_{L-k} \right)=A+B+C+D$. Let us show the following equations for all $k=2,\dots ,L-1$,
\begin{align}
\label{eq:A} A=&\mathbf{1}_{|\vee y_{i}-\wedge y_{i}|\leq\delta}\left[\left(\min(\delta,b-\wedge y_{i})\right)^{k}-\left(\vee y_{i}-\wedge y_{i}\right)^{k}\right],\\
\label{eq:B} B=&\mathbf{1}_{|\vee y_{i}-\wedge y_{i}|\leq\delta}\left[\left(\min(\delta,\vee y_{i}-a)\right)^{k}-\left(\vee y_{i}-\wedge y_{i}\right)^{k}\right],\\
\label{eq:C} C=&\mathbf{1}_{|\vee y_{i}-\wedge y_{i}|\leq\delta}\left(\vee y_{i}-\wedge y_{i}\right)^{k},\\
\label{eq:D} D  = & \mathbf{1}_{|\vee y_{i}-\wedge y_{i}|\leq\delta}\left[\left(\vee y_{i}-\wedge y_{i}\right)^{k}-\left(\min(\delta,\vee y_{i}-a)\right)^{k}\right.\\
 & \left.+k\left(\min(\wedge y_{i},b-\delta)-\max\left(\vee y_{i}-\delta,a\right)\right)\delta^{k-1}+\left(\max(\delta,b-\wedge y_{i})\right)^{k}-\left(b-\wedge y_{i}\right)^{k}\right] \nonumber
\end{align}
and
\begin{equation}\label{eq:Sigma}
\Sigma\left( y_{1},\dots ,y_{L-k} \right)=\mathbf{1}_{|\vee y_{i}-\wedge y_{i}|\leq\delta}\left[\left(k+1\right)\delta^{k}+k\left(\min(\wedge y_{i},b-\delta)-\max\left(\vee y_{i},a+\delta\right)\right)\delta^{k-1}\right].
\end{equation}
Let us fix some $k$ in $\{2,\dots,L-1\}$.
\paragraph{Proof of~\eqref{eq:A}}
To compute $A$, it is sufficient to consider the
case when $x_{1}=\vee x_{i}$, provided a multiplication by $k$, hence
\begin{eqnarray*}
A & = &  k\intop_{x_{1}=\vee y_{i}}^{b}\left(\intop_{\left[\wedge y_{i},x_{1}\right]^{k-1}}\mathbf{1}_{|x_{1}-\wedge y_{i}|\leq\delta}\ dx_{2}\ldots dx_{k}\right)dx_{1}\\
 & = &  k\mathbf{1}_{|\vee y_{i}-\wedge y_{i}|\leq\delta}\intop_{x_{1}=\vee y_{i}}^{\min(\wedge y_{i}+\delta,b)}\left(\intop_{\left[\wedge y_{i},x_{1}\right]^{k-1}}1\ dx_{2}\ldots dx_{k}\right)dx_{1}\\
 & = &  k\mathbf{1}_{|\vee y_{i}-\wedge y_{i}|\leq\delta}\intop_{x_{1}=\vee y_{i}}^{\min(\wedge y_{i}+\delta,b)}\left(x_{1}-\wedge y_{i}\right)^{k-1}\, dx_{1}\\
 & = & \mathbf{1}_{|\vee y_{i}-\wedge y_{i}|\leq\delta}\left[\left(\min(\wedge y_{i}+\delta,b)-\wedge y_{i}\right)^{k}-\left(\vee y_{i}-\wedge y_{i}\right)^{k}\right]\\
 & = & \mathbf{1}_{|\vee y_{i}-\wedge y_{i}|\leq\delta}\left[\left(\min(\delta,b-\wedge y_{i})\right)^{k}-\left(\vee y_{i}-\wedge y_{i}\right)^{k}\right].
\end{eqnarray*}
\paragraph{Proof of~\eqref{eq:B}}
To calculate $B$, we use the same idea and consider the case when $x_{1}=\wedge x_{i}$, leading to
\begin{eqnarray*}
B & = &  k\intop_{x_{1}=\max\left(\vee y_{i}-\delta,a\right)}^{\wedge y_{i}}\left(\vee y_{i}-x_{1}\right)^{k-1}\, dx_{1}\\
 & = & \mathbf{1}_{|\vee y_{i}-\wedge y_{i}|\leq\delta}\left[\left(\vee y_{i}-\max\left(\vee y_{i}-\delta,a\right)\right)^{k}-\left(\vee y_{i}-\wedge y_{i}\right)^{k}\right]\\
 & = & \mathbf{1}_{|\vee y_{i}-\wedge y_{i}|\leq\delta}\left[\left(\min(\delta,\vee y_{i}-a)\right)^{k}-\left(\vee y_{i}-\wedge y_{i}\right)^{k}\right].
\end{eqnarray*}
\paragraph{Proof of~\eqref{eq:C}} This case is pretty clear.
\[
C=\intop_{\left[\wedge y_{i},\vee y_{i}\right]^{k}}\mathbf{1}_{\left|\vee y_{i}-\wedge y_{i}\right|\leq\delta}\ dx_{1}\ldots dx_{k}=\mathbf{1}_{|\vee y_{i}-\wedge y_{i}|\leq\delta}\left(\vee y_{i}-\wedge y_{i}\right)^{k}
\]
\paragraph{Proof of~\eqref{eq:D}} To calculate $D$, it is sufficient to consider the case when $x_{1}=\wedge x_{i}$
and $x_{2}=\vee x_{i}$, provided a multiplication by $k\left(k-1\right)$, hence
\begin{eqnarray*}
D & = &  k\left(k-1\right)\intop_{x_{1}=a}^{\wedge y_{i}}\intop_{x_{2}=\vee y_{i}}^{b}\left(\intop_{x_{1}}^{x_{2}}\mathbf{1}_{|x_{2}-x_{1}|\leq\delta}\ dx_{3}\ldots dx_{k}\right)dx_{2}dx_{1}\\
 & = &  k\left(k-1\right)\mathbf{1}_{|\vee y_{i}-\wedge y_{i}|\leq\delta}\intop_{x_{1}=\max\left(\vee y_{i}-\delta,a\right)}^{\wedge y_{i}}\intop_{x_{2}=\vee y_{i}}^{\min(x_{1}+\delta,b)}\left(x_{2}-x_{1}\right)^{k-2}\, dx_{2}dx_{1}\\
 & = &  k\mathbf{1}_{|\vee y_{i}-\wedge y_{i}|\leq\delta}\intop_{x_{1}=\max\left(\vee y_{i}-\delta,a\right)}^{\wedge y_{i}}\left(\min(x_{1}+\delta,b)-x_{1}\right)^{k-1}-\left(\vee y_{i}-x_{1}\right)^{k-1}\, dx_{1}\\
 & = & \mathbf{1}_{|\vee y_{i}-\wedge y_{i}|\leq\delta}\left[\left(\vee y_{i}-\wedge y_{i}\right)^{k}-\left(\vee y_{i}-\max\left(\vee y_{i}-\delta,a\right)\right)^{k}\right.\\
 &  & \left.+k\left(\min(\wedge y_{i},b-\delta)-\max\left(\vee y_{i}-\delta,a\right)\right)\delta^{k-1}+\left(b-\min(\wedge y_{i},b-\delta)\right)^{k}-\left(b-\wedge y_{i}\right)^{k}\right]\\
 & = & \mathbf{1}_{|\vee y_{i}-\wedge y_{i}|\leq\delta}\left[\left(\vee y_{i}-\wedge y_{i}\right)^{k}-\left(\min(\delta,\vee y_{i}-a)\right)^{k}\right.\\
 &  & \left.+k\left(\min(\wedge y_{i},b-\delta)-\max\left(\vee y_{i}-\delta,a\right)\right)\delta^{k-1}+\left(\max(\delta,b-\wedge y_{i})\right)^{k}-\left(b-\wedge y_{i}\right)^{k}\right].
\end{eqnarray*}
\paragraph{Proof of~\eqref{eq:Sigma}} Remark that
\begin{equation}\label{eq:relation:min}
\left(\min(\delta,b-\wedge y_{i})\right)^{k}+\left(\max(\delta,b-\wedge y_{i})\right)^{k}=\delta^{k}+\left(b-\wedge y_{i}\right)^{k}
\end{equation}
and 
\begin{equation}\label{eq:relation:max}
\max\left(\vee y_{i}-\delta,a\right)=\max\left(\vee y_{i},a+\delta\right)-\delta.
\end{equation}
Gathering~\eqref{eq:A},~\eqref{eq:B},~\eqref{eq:C},~\eqref{eq:D},~\eqref{eq:relation:min} and~\eqref{eq:relation:max} gives~\eqref{eq:Sigma}.
Hence, Equation~\eqref{eq:Sigma} holds for every $k$ in $\{2,\dots,L-1\}$. 

Moreover, if $k=0$, then $\Sigma\left( y_{1},\dots ,y_{L-k} \right)=\mathbf{1}_{|\vee y_{i}-\wedge y_{i}|\leq\delta}$ and, if $k=1$, then $\Sigma\left( y_{1},\dots ,y_{L-k} \right)=\mathbf{1}_{|\vee y_{i}-\wedge y_{i}|\leq\delta}\left[\min(\wedge y_{i}+\delta,b)-\max\left(\vee y_{i}-\delta,a\right)\right]$. To summarize, Equation~\eqref{eq:Sigma} holds for every $k$ in $\{0,\dots,L-1\}$.

It remains to compute $$I(L,k)=\intop\limits _{\left[a,b\right]^{L-k}}\Sigma\left(y_{1},\ldots,y_{L-k}\right)^{2}\, dy_{1}\ldots dy_{L-k}.$$
In order to do that, let us decompose the integral with respect to the following conditions on $(y_{1},\dots ,y_{L-k})$:
\begin{enumerate}
\item $\vee y_{i}< a+\delta$. In this case, $\Sigma=\delta^{k-1}\left[\delta+k\left(\wedge y_{i}-a\right)\right]$,
and denote the integral $Y$.
\item $\wedge y_{i}>b-\delta$. In this case, $\Sigma=\delta^{k-1}\left[\delta+k\left(b-\vee y_{i}\right)\right]$,
and denote the integral $Z$.
\item $\vee y_{i}> a+\delta$ and $\wedge y_{i}< b-\delta$. In this
case, $\Sigma=\mathbf{1}_{|\vee y_{i}-\wedge y_{i}|\leq\delta}\delta^{k-1}\left[\left(k+1\right)\delta-k\left(\vee y_{i}-\wedge y_{i}\right)\right]$,
and denote the integral $W$.
\end{enumerate}
These three cases are distinct because $\delta<(b-a)/2$, so we have partitioned $\left[a,b\right]^{L-k}$ up to a null measure set and $I(L,k)=Y+Z+W$.
Let us show the following equations for all $k=0,\dots ,L-2$,
\begin{align}
\label{eq:Y:Z} Y=Z=& \,C\left(L,k\right)\delta^{L+k},\\
\label{eq:W} W=& f\left(L,k\right)\left(b-a\right)\delta^{L+k-1}-\left[f\left(L,k\right)+g\left(L,k\right)\right]\delta^{L+k},
\end{align}
where
\begin{equation}\label{eq:def:C:n:k}
C\left(L,k\right)=\left(L-k\right)\frac{\left(k+1\right)^{L-k+2}}{k^{L-k}}\int_{0}^{\frac{k}{k+1}}t^{L-k-1}\left(1-t\right)^{2}\, dt,
\end{equation}
\[
f\left(L,k\right)=\left(L-k\right)\left(k+1\right)^{2}-2\left(L-k-1\right)k\left(k+1\right)+\frac{\left(L-k\right)\left(L-k-1\right)}{\left(L-k+1\right)}k^{2}
\]
and
\[
g\left(L,k\right)=\left(k+1\right)^{2}-2\frac{\left(L-k-1\right)k\left(k+1\right)}{\left(L-k+1\right)}+\frac{\left(L-k\right)\left(L-k-1\right)}{\left(L-k+1\right)\left(L-k+2\right)}k^{2}.
\]

Let us fix some $k$ in $\{0,\dots,L-2\}$.
\paragraph{Proof of~\eqref{eq:Y:Z}}
To compute $Y$, it is sufficient to consider the
case when  $y_{1}=\wedge y_{i}$, provided a multiplication by $\left(L-k\right)$, hence
\begin{eqnarray*}
Y & = &  \intop_{\vee y_{i}\leq a+\delta}\Sigma(y_{1},\dots ,y_{L-k})^{2}\ dy_{1}\ldots dy_{L-k}\\
 & = &  \left(L-k\right)\delta^{2k-2}\intop_{y_{1}=a}^{a+\delta}\left(\intop_{\left[y_{1},a+\delta\right]^{L-k-1}}\left[\delta+k\left(y_{1}-a\right)\right]^{2}\, dy_{2}\ldots dy_{L-k}\right)dy_{1}\\
 & = &  \left(L-k\right)\delta^{2k-2}\intop_{y_{1}=a}^{a+\delta}\left(a+\delta-y_{1}\right)^{L-k-1}\left[\delta+k\left(y_{1}-a\right)\right]^{2}\, dy_{1}.
\end{eqnarray*}
Defining the variable $u=a+\delta-y_{1}$ leads to
\begin{eqnarray*}
Y & = &  \left(L-k\right)\delta^{2k-2}\int_{0}^{\delta}u^{L-k-1}\left[\delta+k\left(\delta-u\right)\right]^{2}\, du\\
 & = &  \left(L-k\right)\delta^{2k-2}\int_{0}^{\delta}u^{L-k-1}\left[\left(k+1\right)\delta-ku\right]^{2}\, du,
\end{eqnarray*}
and by defining the variable $t=\frac{ku}{\left(k+1\right)\delta}$ we have
\begin{eqnarray*}
Y & = &  \left(L-k\right)\delta^{2k-2}\int_{0}^{\frac{k}{k+1}}\left(\frac{\left(k+1\right)\delta t}{k}\right)^{L-k-1}\left(k+1\right)^{2}\delta^{2}\left(1-t\right)^{2}\frac{\left(k+1\right)\delta}{k}dt\\
 & = &  \left(L-k\right)\delta^{L+k}\frac{\left(k+1\right)^{L-k+2}}{k^{L-k}}\int_{0}^{\frac{k}{k+1}}t^{L-k-1}\left(1-t\right)^{2}dt.
\end{eqnarray*}
The computation of $Z$ can be done in the same way by inverting the roles of $a$ and $b$ on the one hand and the roles of $\wedge y_{i}$ and $\vee y_{i}$ on the other hand. This leads to $Z=Y$ and Equation~\eqref{eq:Y:Z}.

\paragraph{Proof of~\eqref{eq:W}}
To compute $W$, it is sufficient to consider the
case when  $y_{1}=\wedge y_{i}$ and $y_{2}=\vee y_{i}$, provided a multiplication by $\left(L-k\right)\left(L-k-1\right)$, hence
\begin{eqnarray*}
W & = &  \left(L-k\right)\left(L-k-1\right)\delta^{2k-2}\\
 &  & \quad \intop_{y_{1}=a}^{b-\delta}\intop_{y_{2}=\max\left(y_{1},a+\delta\right)}^{b}
 \left(\int\mathbf{1}_{|y_{2}-y_{1}|\leq\delta}\left[\left(k+1\right)\delta-k\left(y_{2}-y_{1}\right)\right]^{2}\, dy_{3}\ldots dy_{L-k}\right)dy_{2}dy_{1}\\
 & = &  \left(L-k\right)\left(L-k-1\right)\delta^{2k-2}\\
 &  &  \quad \intop_{y_{1}=a}^{b-\delta}\intop_{y_{2}=\max\left(y_{1},a+\delta\right)}^{y_{1}+\delta}
 \left(\left(y_{2}-y_{1}\right)^{L-k-2}
 \left[\left(k+1\right)^{2}\delta^{2}-2k
 \left(k+1\right)\delta
 \left(y_{2}-y_{1}\right)\right.\right.\\
 & &  \quad \hphantom{\intop_{y_{1}=a}^{b-\delta}\intop_{y_{2}=\max\left(y_{1},a+\delta\right)}^{y_{1}+\delta}} \left. \left. +k^{2} \left(y_{2}-y_{1}\right)^{2}\right]\right)\,  dy_{2}dy_{1},
\end{eqnarray*}
which leads to
\begin{eqnarray*}
W & = & \int_{a}^{b-\delta} \left\{ \left(L-k\right)\delta^{2k-2} \left(k+1\right)^{2}\delta^{2}\left[\delta^{L-k-1}-\left(\max\left(y_{1},a+\delta\right)-y_{1}\right)^{L-k-1}\right] \right\}\\
 &  &  - \left\{ 2 \left(L-k-1\right)\delta^{2k-2} k\left(k+1\right)\delta\left[\delta^{L-k}-\left(\max\left(y_{1},a+\delta\right)-y_{1}\right)^{L-k}\right] \right\}\\
 &  &  + \left\{ \frac{\left(L-k\right)\left(L-k-1\right)}{\left(L-k+1\right)}\delta^{2k-2}k^{2}\left[\delta^{L-k+1}-\left(\max\left(y_{1},a+\delta\right)-y_{1}\right)^{L-k+1}\right]\right\} \, dy_{1}\\
 & = & W_{1} + W_{2},
\end{eqnarray*}
where $W_{1}$ (resp. $W_{2}$) denotes the integral between $a$ and $a+\delta$ (resp. between $a+\delta$ and $b-\delta$). Let us denote 
\begin{equation}\label{eq:def:f:n:k}
f\left(L,k\right)=\left(L-k\right)\left(k+1\right)^{2}-2\left(L-k-1\right)k\left(k+1\right)+\frac{\left(L-k\right)\left(L-k-1\right)}{\left(L-k+1\right)}k^{2}.
\end{equation} 
Then, on the one hand
\begin{eqnarray*}
W_{1} & = &  \delta^{2k-2}\left[\int_{a}^{a+\delta}f\left(L,k\right)\delta^{L-k+1}\, dy_{1}\right.\\
 &  & -\left(L-k\right)\left(k+1\right)^{2}\delta^{2}\int_{a}^{a+\delta}\left(a+\delta-y_{1}\right)^{L-k-1}\, dy_{1}\\
 &  &  +2\left(L-k-1\right)k\left(k+1\right)\delta\int_{a}^{a+\delta}\left(a+\delta-y_{1}\right)^{L-k}\, dy_{1}\\
 &  & \left. -\frac{\left(L-k\right)\left(L-k-1\right)}{\left(L-k+1\right)}k^{2}\int_{a}^{a+\delta}\left(a+\delta-y_{1}\right)^{L-k+1}\, dy_{1}\right]\\
 & = &  f\left(L,k\right)\delta^{L+k}-g\left(L,k\right)\delta^{L+k},
\end{eqnarray*}
with 
\begin{equation}\label{eq:def:g:n:k}
g\left(L,k\right)=\left(k+1\right)^{2}-2\frac{\left(L-k-1\right)k\left(k+1\right)}{\left(L-k+1\right)}+\frac{\left(L-k\right)\left(L-k-1\right)}{\left(L-k+1\right)\left(L-k+2\right)}k^{2}.
\end{equation}
On the other hand,
\begin{eqnarray*}
W_{2} & = &  \delta^{2k-2}\int_{a+\delta}^{b-\delta}\left[\left(L-k\right)\left(k+1\right)^{2}-2\left(L-k-1\right)k\left(k+1\right)\right.\\
& & \left. \hphantom{\delta^{2k-2}\int_{a+\delta}^{b-\delta} \left(L-k\right)\left(k+1\right)^{2}} \quad \quad +\frac{\left(L-k\right)\left(L-k-1\right)}{\left(L-k+1\right)}k^{2}\right]\delta^{L-k+1}\, dy_{1}\\
 & = & \left(b-a-2\delta\right)f\left(L,k\right)\delta^{L+k-1}\\
 & = & f\left(L,k\right)\left(b-a\right)\delta^{L+k-1}-2f\left(L,k\right)\delta^{L+k}
\end{eqnarray*}
where $f\left(L,k\right)$ is defined by~\eqref{eq:def:f:n:k}. Then,~\eqref{eq:W} clearly follows from $W=W_{1}+W_{2}$.

Moreover, if $k=L-1$, then Equations~\eqref{eq:Y:Z} and~\eqref{eq:W} are still valid. To summarize, Equations~\eqref{eq:Y:Z} and~\eqref{eq:W} hold true for every $k$ in $\{0,\dots,L-1\}$.

Gathering~\eqref{eq:Y:Z} and~\eqref{eq:W} yields
\begin{equation*}
I(L,k) = f\left(L,k\right)\left(b-a\right)\delta^{L+k-1}-\left[f\left(L,k\right)+g\left(L,k\right)-2C\left(L,k\right)\right]\delta^{L+k},
\end{equation*}
for every $k$ in $\{0,\dots ,L-1\}$.\\

To conclude, the integral involved in~\eqref{eq:def:C:n:k} can be computed with respect to $k$ and $L$ in the following way,
\[
\int_{0}^{\frac{k}{k+1}}t^{L-k-1}\left(1-t\right)^{2}\, dt=\left(\frac{k}{k+1}\right)^{L-k}\left[\frac{1}{L-k}-\frac{2k}{\left(k+1\right)\left(L-k+1\right)}+\frac{k^{2}}{\left(k+1\right)^{2}\left(L-k+2\right)}\right].
\]
Moreover, in the result stated in Proposition~\ref{prop:Formule des I(k)} we just used the software Mathematica in order to simplify
the expressions. These simplifications lead to
\[
f\left(L,k\right)=\frac{k\left(k+1\right)+L\left(L+1\right)}{L-k+1}
\]
and
\[
h\left(L,k\right):=f\left(L,k\right)+g\left(L,k\right)-2C\left(L,k\right)=\frac{-k^{3}+k^{2}(2+L)+k(5+2L-L^{2})+L^{3}+2L^{2}-L-2}{(L-k+2)(L-k+1)}.
\]

\newpage{}

\bibliographystyle{apalike}
\bibliography{references}

\end{document}